\documentclass[a4paper,reqno]{amsart}

\usepackage{etoolbox}

\newtoggle{default}\toggletrue{default}
\newtoggle{siam}\togglefalse{siam}
\newtoggle{ams}\togglefalse{ams}
\newtoggle{wiley}\togglefalse{wiley}
\newtoggle{birk}\togglefalse{birk}
\newtoggle{springer}\togglefalse{springer}

\newtoggle{endwithsymbol}\togglefalse{endwithsymbol}

\newtoggle{boldproof}\togglefalse{boldproof}

\newtoggle{nolinkcolor}\togglefalse{nolinkcolor}

\usepackage{amsmath,amssymb,amsfonts,amsthm}
\usepackage{mathtools,thmtools}
\usepackage[T1]{fontenc}

\usepackage{enumitem}
\setlist[enumerate,1]{leftmargin=*,label=\textup{(\roman{*})}}
\setlist[itemize,1]{leftmargin=15pt}

\newlist{steps}{enumerate}{1}
\setlist[steps]{label=\textup{\textbf{\arabic{*}.~Step:}}, ref=\textup{\arabic{*}.~Step}, align=left, leftmargin=0pt, itemindent=*,labelindent=0pt, labelsep=3pt, itemsep=3pt, parsep=2pt,topsep=2pt}

\usepackage{multistyle}

\usepackage{calc}
\usepackage{microtype} 
\usepackage{bm}
\usepackage{wrapfig}
\usepackage{booktabs}
\usepackage{multirow}
\usepackage{arydshln}

\usepackage{hyperref} 
\ifboolexpr{togl{default} or togl{birk}}{
  \hypersetup{%
    colorlinks=true,
    linkcolor=blue,
    citecolor=green!60!black,
  }
}{}
\usepackage{cleveref}

\usepackage{orcidlink} 

\usepackage{pifont}
\usepackage{upgreek}

\usepackage{subcaption}
\usepackage{suffix}

\ifboolexpr{not togl{springer}}{\usepackage{tikz}}{}
\usetikzlibrary{cd}
\usetikzlibrary{calc}
\usetikzlibrary{arrows}

\usepackage[europeanresistors,smartlabels]{circuitikz}

\ifboolexpr{not togl{springer}}{\usepackage{ifthen}}{}

\graphicspath{{./figs/}}

\makeatletter
\pgfcircdeclarebipole{}
    {\ctikzvalof{bipoles/tline/height}}{tline}
    {\ctikzvalof{bipoles/tline/height}}
    {\ctikzvalof{bipoles/tline/width}}
{
    \pgfpointdiff{\pgfpointanchor{\ctikzvalof{bipole/name}start}{center}}
                 {\pgfpointanchor{\ctikzvalof{bipole/name}end}{center}}
    \pgfmathparse{veclen(\the\pgf@x,\the\pgf@y)}
    \pgftransformxshift{\pgfmathresult/2-30pt}
    \pgf@circ@res@left=\dimexpr-\pgfmathresult pt+40pt\relax
    \pgf@circ@res@step=.2\pgf@circ@res@right 
    \pgfsetlinewidth{\pgfkeysvalueof{/tikz/circuitikz/bipoles/thickness}\pgfstartlinewidth}
    \pgfpathellipse{\pgfpoint{\pgf@circ@res@right-\pgf@circ@res@step}{0}}
                   {\pgfpoint{\pgf@circ@res@step}{0}}
                   {\pgfpoint{0}{-\pgf@circ@res@up}}
    \pgfpathmoveto{\pgfpoint{\pgf@circ@res@right-\pgf@circ@res@step}{\pgf@circ@res@up}}
    \pgfpathlineto{\pgfpoint{\pgf@circ@res@left+\pgf@circ@res@step}{\pgf@circ@res@up}}
    \pgfpatharc{-90}{90}{-\pgf@circ@res@step and -\pgf@circ@res@up}
    \pgfpathlineto{\pgfpoint{\pgf@circ@res@right-\pgf@circ@res@step}{\pgf@circ@res@down}}
    \pgfsetfillcolor{white}
    \pgfusepath{draw,fill}
    \pgfpathmoveto{\pgfpoint{\pgf@circ@res@right-2*\pgf@circ@res@step}{0pt}}
    \pgfpathlineto{\pgfpoint{\pgf@circ@res@right}{0pt}}
    \pgfusepath{draw}
}

\makeatletter
\def\moverlay{\mathpalette\mov@rlay}
\def\mov@rlay#1#2{\leavevmode\vtop{%
   \baselineskip\z@skip \lineskiplimit-\maxdimen
   \ialign{\hfil$\m@th#1##$\hfil\cr#2\crcr}}}
\newcommand{\charfusion}[3][\mathord]{
    #1{\ifx#1\mathop\vphantom{#2}\fi
        \mathpalette\mov@rlay{#2\cr#3}
      }
    \ifx#1\mathop\expandafter\displaylimits\fi}
\makeatother

\newcommand{\cupdot}{\charfusion[\mathbin]{\cup}{\cdot}}
\newcommand{\bigcupdot}{\charfusion[\mathop]{\bigcup}{\cdot}}

\newcounter{i} 
\newtoks\striche 

\newcommand{\dd}{\mathrm{d}} 
\newcommand{\dx}[1][x]{\mathop{\dd#1}}

\newcommand{\cl}[2][]{\overline{#2}\ifthenelse{ \equal{#1}{} }{}{^{#1}}} 
\newcommand{\conj}[1]{\overline{#1}} 

\newcommand{\argdot}{\boldsymbol{\cdot}}

\DeclareMathOperator{\id}{id}

\renewcommand{\div}{\operatorname{div}}
\newcommand{\grad}{\nabla}
\DeclareMathOperator{\rot}{rot}

\DeclarePairedDelimiter{\sset}{\lbrace}{\rbrace}
\DeclarePairedDelimiter{\norm}{\lVert}{\rVert}
\DeclarePairedDelimiter{\abs}{\vert}{\vert}

\DeclarePairedDelimiterX{\dset}[2]{\lbrace}{\rbrace}{#1\,\delimsize\vert\,\mathopen{} #2}
\DeclarePairedDelimiterX{\scprod}[2]{\langle}{\rangle}{#1,#2}
\DeclarePairedDelimiterX{\dualprod}[2]{\langle}{\rangle}{#1,#2}
\DeclarePairedDelimiterX{\sdprod}[2]{\llangle}{\rrangle}{#1,#2} 

\renewcommand{\Re}{\operatorname{Re}}

\newcommand{\adjunsymb}{\ast} 

\newcommand{\adjun}[1][1]{%
  \setcounter{i}{1}%
  \striche={\adjunsymb}%
  \loop%
  \ifnum\value{i}<#1%
  \striche=\expandafter{\the\expandafter\striche\adjunsymb}%
  \stepcounter{i}%
  \repeat%
  ^{\the\striche}%
}

\newcommand{\mapping}[4]{%
  \left\{%
    \begin{array}{rcl}%
      #1 &\to & #2, \\
      #3 &\mapsto & #4
    \end{array}%
  \right.%
}

\newcommand{\trans}{^{\mathsf{T}}}

\newcommand{\Lp}[1]{\mathrm{L}^{#1}} 

\newcommand{\conC}{\mathrm{C}} 

\newcommand{\boundtr}[1][]{\gamma_{0}\ifthenelse{\equal{#1}{}}{}{\big\vert_{#1}}}
\newcommand{\normaltr}[1][]{\gamma_{\nu}\ifthenelse{\equal{#1}{}}{}{\big\vert_{#1}}}
\newcommand{\tantr}[1][]{\pi_{\tau}\ifthenelse{\equal{#1}{}}{}{\big\vert_{#1}}}
\newcommand{\tanxtr}[1][]{\gamma_{\tau}\ifthenelse{\equal{#1}{}}{}{\big\vert_{#1}}}

\newcommand{\portOp}{P}
\newcommand{\wcur}[1][]{w_{\ifthenelse{\equal{#1}{}}{}{#1,}\textup{cur}}}

\newcommand{\N}{\mathbb{N}}
\iftoggle{springer}{

\renewcommand{\R}{\mathbb{R}}
\renewcommand{\C}{\mathbb{C}}
}{

\newcommand{\R}{\mathbb{R}}
\newcommand{\C}{\mathbb{C}}
}

\DeclareMathOperator{\rk}{rank}

\newcommand{\ddt}{\frac{\text{\normalfont d}}{\text{\normalfont d}t}}
\newcommand{\ddts}{\tfrac{\text{\normalfont d}}{\text{\normalfont d}t}}

\DeclareMathOperator{\ext}{ext}

\newcommand{\myendhere}{\iftoggle{endwithsymbol}{\qedhere}{}}

\iftoggle{siam}{

\newsiamremark{remark}{Remark}
\newsiamremark{hypothesis}{Hypothesis}
\crefname{hypothesis}{Hypothesis}{Hypotheses}
\newsiamthm{claim}{Claim}

\newcommand{\qedhere}{}

}{}

\makeatletter
\iftoggle{wiley}{

\let\c@theorem\relax

\theoremstyle{WBstyleone}
\newtheorem{theorem}{Theorem}[section]
\newtheorem{lemma}[theorem]{Lemma}
\newtheorem{corollary}[theorem]{Corollary}
\newtheorem{proposition}[theorem]{Proposition}

\theoremstyle{WBstyletwo}
\newtheorem{remark}[theorem]{Remark}

\theoremstyle{WBstyletwo}
\newtheorem{claim}{Claim}

\theoremstyle{WBstylethree}
\newtheorem{definition}[theorem]{Definition}
\newtheorem{assumption}[theorem]{Assumption}

}{}
\makeatother

\ifboolexpr{togl{default} or togl{ams} or togl{birk}}{

\theoremstyle{plain}
\newtheorem{theorem}{Theorem}[section]
\newtheorem{lemma}[theorem]{Lemma}
\newtheorem{proposition}[theorem]{Proposition}
\newtheorem{corollary}[theorem]{Corollary}

\iftoggle{endwithsymbol}{%

    \declaretheorem[style=definition,sibling=theorem,qed=\ding{71},name=Definition]{definition}
    \declaretheorem[style=definition,sibling=theorem,qed=\ding{71},name=Example]{example}
    \declaretheorem[style=definition,sibling=theorem,qed=\ding{71},name=Assumption]{assumption}

    \iftoggle{boldproof}{%
    \declaretheorem[style=definition,sibling=theorem,qed=\ding{71},name=Remark]{remark}
    }{%
    \declaretheorem[style=remark,sibling=theorem,qed=\ding{71},name=Remark]{remark}
    }
  }%
  {%
    \theoremstyle{definition}
    \newtheorem{definition}[theorem]{Definition}
    
    \newtheorem{assumption}[theorem]{Assumption}

    \iftoggle{boldproof}{}{\theoremstyle{remark}}
    \newtheorem{remark}[theorem]{Remark}
  }%

\theoremstyle{definition}

\newtheorem{claim}{Claim}

}{}

\iftoggle{springer}{
}{
\newenvironment{appendices}{\appendix}{}
}

\makeatletter
\iftoggle{boldproof}{
  \providecommand{\proofNameStyle}{\bfseries}
  \renewenvironment{proof}[1][\proofname]{\par
    \pushQED{\qed}%
    \normalfont \topsep6\p@\@plus6\p@\relax
    \trivlist
    \item[\hskip\labelsep\proofNameStyle
    #1\@addpunct{.}]\ignorespaces
  }{%
    \popQED\endtrivlist\@endpefalse
  }
}{}
\makeatother

\begin{document}

\mytitle[Coupled transmission line and Maxwell systems]{Modeling of radiating curved cables via coupled telegrapher's and Maxwell's equations}

\mykeywords{Maxwell's equations, transmission line, power balance, coupled systems}

\mymsccodes{35Q61, 78A25, 78A40}

\begin{myabstract}
  We investigate the electromagnetic interactions of cable harnesses in the time domain. We present a novel model that allows for curved cables, extending the standard assumptions typically made in transmission line modeling. The cables are described by the telegrapher's equations, the classical model for transmission lines, driven by input signals implemented through appropriate boundary conditions, such as imposed voltages at cable ends.
  The cables interact via electromagnetic radiation; the latter is determined by Maxwell’s equations. This interaction is incorporated into the model through boundary conditions imposed on the electromagnetic field. The resulting coupling between the transmission lines and Maxwell's equations is energetically consistent. In particular, we show that the coupled system satisfies a global power balance.



\end{myabstract}


\myauthor[M.~Clemens]{Markus Clemens}
\myorcid{0000-0002-1226-7840}
\myemail{clemens@uni-wuppertal.de}
\myorgdiv{Chair of Electromagnetic Theory}
\myorgname{University of Wuppertal}
\myorgstreet{Rainer-Gruenter-Stra{\ss}e 21}
\myorgpostcode{42119}
\myorgcity{Wuppertal}
\myorgstate{North Rhine-Westphalia}
\myorgcountry{Germany}

\myauthor[M.~G\"unther]{Michael G\"unther}
\myorcid{0000-0002-2195-4300}
\myemail{guenther@uni-wuppertal.de}
\myorgdiv{Chair of Applied Mathematics}
\myorgname{University of Wuppertal}
\myorgstreet{Gau{\ss}stra{\ss}e 22}
\myorgpostcode{42119}
\myorgcity{Wuppertal}
\myorgstate{North Rhine-Westphalia}
\myorgcountry{Germany}

\myauthor[T.~Reis]{Timo Reis}
\myorcid{0000-0003-0721-8494}
\myemail{timo.reis@tu-ilmenau.de}
\myorgdiv{Chair of Systems Theory and Partial Differential Equations}
\myorgname{Technische Universit{\"a}t Ilmenau}
\myorgstreet{Weimarer Str.\ 25}
\myorgpostcode{98693}
\myorgcity{Ilmenau}
\myorgstate{Thuringia}
\myorgcountry{Germany}

\myauthor[N.~Skrepek]{Nathanael Skrepek}
\myorcid{0000-0002-3096-4818}
\myemail{n.skrepek@utwente.nl}
\myorgdiv{Department of Applied Mathematics}
\myorgname{University of Twente}
\myorgstreet{P.O.\ Box 217}
\myorgpostcode{7500 AE}
\myorgcity{Enschede}
\myorgstate{Overijssel}
\myorgcountry{The Netherlands}

\date{\today}
\mymaketitle

\section{Introduction}

Though all electromagnetic effects are described by Maxwell's equations, there are a~variety of situations where simplified models result in an acceptable picture of reality. One of these is the transmission lines, which involves current and voltage distributions along spatially one-dimensional cables (for this reason, we use the terms ``cable'' and ``transmission line'' interchangeably throughout this article). These are modelled by the telegrapher's equations, a~hyperbolic partial differential equation in one spatial dimension.
This model reflects physical effects like crosstalk, time delay and energy loss, which typically occur when comparatively long cables are driven with high-frequency voltages and currents \cite{Kr84,MaSk87,Ha89}.
Transmission lines, however, typically do not form a closed physical system; instead, they may interact with an electromagnetic field in a bidirectional manner: The voltages and currents along the cables cause electromagnetic radiation and, vice-versa, an electromagnetic field excites voltages and currents along the cables.
Overall, one obtains a~model consisting of coupled telegrapher's and Maxwell's equations.
For this type of problem there exists a rich literature in electrical engineering such as, to mention only a few, \cite{RuRaPaRe02,LiWaYaKaAlChFa17,Ra12,LaNuTe88,PaAb81,Agrawal1980}.
In all of these works, the cables are assumed to be straight. In contrast, we allow for curved cables, while still assuming a constant circular cross-section. We consider $k$ cables interacting with the surrounding electromagnetic field. The dynamics of voltages and currents along the cables are extended to their lateral surfaces, which form two-dimensional manifolds. These quantities serve as tangential boundary values for the electric and magnetic field intensities governed by Maxwell's equations, which describe the electromagnetic field outside the cables.
The inputs and outputs of the overall system are given by the boundary values of the telegrapher's equations, i.e., the voltages and currents at the ends of the cables. The following table illustrates the mixed-dimensional nature of the problem by summarizing the physical quantities involved in the system, categorized by spatial dimension and their roles in the system-theoretic context.

\begin{table}[h]
\centering
\renewcommand{\arraystretch}{1.1}
\begin{tabular}{|l:p{5cm}:p{3cm}|}
\hline
\textbf{Spatial} & \multirow{2}{*}{\textbf{Quantity}} & \multirow{2}{*}{\textbf{Role}} \\
\textbf{Dimension} & & \\
\hdashline
\qquad\parbox[0pt][2.5em][c]{0cm}{0} & \parbox[0pt][3em][c]{4.8cm}{voltages/currents at cable ends} & inputs and outputs \\
\hdashline
\qquad\parbox[0pt][2.5em][c]{0cm}{1} & \parbox[0pt][3em][c]{4.8cm}{charges/fluxes along cables} & {part of the state} \\
\hdashline
\qquad\parbox[0pt][2.5em][c]{0cm}{2} & \parbox[0pt][3em][c]{4.8cm}{tangential electric/magnetic\\ field at lateral cable surfaces} & coupling quantities \\

\hdashline
\qquad\parbox[0pt][2.5em][c]{0cm}{3} & \parbox[0pt][3em][c]{4.8cm}{electric/magnetic flux density in the electromagnetic field} & part of the state \\
\hline
\end{tabular}
\end{table}

Note that the transmission line model is one-dimensional, which leads to a dimensional mismatch with the coupling quantities. We address this mismatch using suitable lifting operators, in a manner similar to \cite{JaSkEh23}.


Since the state of the overall system consists of functions of spatial variables, the state space is infinite-dimensional. The inputs and outputs are defined by a class of linear combinations of the boundary voltages and currents of the transmission lines. Consequently, both the input and output spaces are finite-dimensional.

This article is organized as follows. After introducing the notation and presenting some basic facts about Maxwell's equations in the remainder of this section, we separately discuss the models for the cables and the electromagnetic field in \Cref{sec:tl} and \Cref{sec:elmag}, respectively. Boundary-controlled telegrapher's equations and Maxwell's equations are treated independently. This part also introduces the input-output configuration, appropriate initial and boundary conditions, and our assumptions on the cable geometry and the computational domain in which the electromagnetic field evolves.
Thereafter, in \Cref{sec:coupling-idea}, we take a closer look at the operators responsible for the coupling. These are the heart of the model, as they lift the 1-D functions from the transmission lines to 2-D functions on the lateral surfaces of the cables, which in turn act as tangential boundary values for the Maxwell equations involving 3-D functions. This will be used to show that the overall system admits a power balance.

\subsection*{Notation and convention}

We use $\id$ for the identity mapping, $\id_n$ stands for the unit matrix of size $n\times{n}$. Further, we write $\scprod{\argdot}{\argdot}_X$ for the inner product in an inner product space $X$, and $A\adjun$ for the adjoint of an operator $A$ acting between inner product spaces. 
Together with the fact that $\C^n$ and $\C^m$ are equipped with the Euclidean inner product, this means that $A\adjun \in\C^{n\times m}$ is the conjugate transpose of $A\in\C^{m\times n}$. Likewise, $x^*$ is the conjugate transpose of $x\in\C^n\cong \C^{n\times 1}$, such that the inner product in $\C^n$ reads
\[
  \scprod{x}{y}_{\C^{n}} = y\adjun x.
\]
For $P\in\C^{n\times n}$, we write $P>0$ ($P\geq0$), if $P=P\adjun$ is positive (semi-)definite. Likewise, 
$P<0$ ($P\leq0$) means that $P=P\adjun$ is negative (semi-)definite. Further, $A^\dagger\in \C^{n\times m}$ denotes the Moore-Penrose inverse of $A\in\C^{m\times n}$.



\subsection*{Prologue: Maxwell's equations} 

Although Maxwell’s equations are well known, we briefly recall them here to fix our notation. These equations form the foundation of all electromagnetic dynamics, including those arising in circuits and transmission lines. In this article, we restrict ourselves to the linear Maxwell equations, which involve the $\R^3$-valued physical quantities
\begin{center}\begin{tabular}{ll}
$\bm{B}$: & magnetic flux density, \\
$\bm{D}$: & electric flux density. \\
\end{tabular}
\end{center}

\noindent The arguments are time $t$ and space $\xi$, where the latter is an element of some given domain $\Omega\subset\R^3$. By writing $\bm{B}(t)$, $\bm{D}(t)$, and $\bm{J}_{\ext}(t)$, we refer to the spatial distributions of the magnetic flux density, electric flux density, and externally applied current density at time $t$, respectively. That is, $\bm{B}(t)$, $\bm{D}(t)$, and $\bm{J}_{\ext}(t)$ are $\R^3$-valued functions defined on $\Omega$.
In the case of linear constitutive relations and linear losses, Maxwell's equations read
\begin{equation}\label{eq:Maxwell}
  \frac{\partial}{\partial t}
  \begin{pmatrix} \bm{B}(t)\\ \bm{D}(t)\end{pmatrix}
  =
  \begin{bmatrix}0 & -\rot \\ \rot &-\bm{\sigma}\end{bmatrix}
  \begin{pmatrix}\bm{\mu}^{-1}\bm{B}(t) \\ \bm{\epsilon}^{-1}\bm{D}(t)\end{pmatrix}+\begin{pmatrix}0 \\ \bm{J}_{\ext}(t)\end{pmatrix}.
\end{equation}
Here, $\bm{\mu}\colon \Omega\to\C^{3\times 3}$ represents magnetic permeability, $\bm{\epsilon}\colon \Omega\to\C^{3\times 3}$ represents electric permittivity, and $\bm{\sigma}\colon \Omega\to\C^{3\times 3}$ corresponds to electric conductivity. The quantities $\bm{H}(t) \coloneqq \bm{\mu}^{-1}\bm{B}(t)$ and $\bm{E}(t) \coloneqq \bm{\epsilon}^{-1}\bm{D}(t)$ are referred to as the magnetic and electric field intensities, resp.
Typically, Maxwell's equations are supplemented with the conditions 
\begin{equation}
\div \bm{B}(t)=0,\quad\div \bm{D}(t)=\bm{\rho}(t),\label{eq:divfree}
\end{equation}
where $\bm{\rho}(t)\colon \Omega\to\C$ is a scalar field representing charge density at time $t$.

\section{Transmission lines}\label{sec:tl}

We now consider the modeling of the cables. They are described by the telegrapher’s equations, the standard model for transmission lines. To account for the interaction with the electromagnetic field at a later stage, these equations are augmented by an additional source term distributed along the transmission line. We then introduce suitable boundary conditions that model the electrical connection at the cable terminals.

\subsection{Telegrapher's equations with distributed excitation}
\noindent We consider $k$ transmission lines, which are modelled by the 
\emph{telegrapher's equations} with additional current excitation and electrical field output. The internal physical quantities are the $\C^{k}$-valued functions
\begin{center}\begin{tabular}{ll}
$\bm{\psi}$: & magnetic flux, \\
$\bm{q}$: & electric charge, \\
\end{tabular}
\end{center}
where each component stands for the flux (resp.\ charge) of one particular transmission line. The functions $\bm{\psi}$ and $\bm{q}$
depend on time $t$ and the spatial variable $\eta\in[0,1]$. As for the variables in Maxwell's equations, we adopt the convention that, for fixed time~$t$, the functions $\bm{\psi}(t), \bm{q}(t) \colon [0,1] \to \R^k$ represent the spatial charge and flux distributions along the transmission line. The system is further excited by an external current $\bm{I}_{\ext}$, and an external electric field intensity $\bm{E}_{\ext}$ is read out. These values, $\bm{E}_{\ext}$ and $\bm{I}_{\ext}$, will later be used to couple the transmission lines with the electromagnetic field. 
 The transmission line model is
\begin{equation}\label{eq:teleq}
  \begin{aligned}
    \ddt
    \begin{pmatrix}
      \bm{\psi}(t) \\ \bm{q}(t)
    \end{pmatrix}
    &=
    \begin{bmatrix}
      -\bm{R} & -\tfrac{\partial}{\partial \eta}\\
      -\tfrac{\partial}{\partial \eta} & -\bm{G}
    \end{bmatrix}
    \begin{bmatrix}
      \bm{L}^{-1} & 0 \\
      0 & \bm{C}^{-1}
    \end{bmatrix}
    \begin{pmatrix}
      \bm{\psi}(t) \\
      \bm{q}(t)
    \end{pmatrix}+
    \begin{bmatrix}
      \phantom{-}0\\-\tfrac{\partial}{\partial \eta}
    \end{bmatrix} \bm{I}_{\ext}(t),\\
    \bm{E}_{\ext}(t)
    &=
    \mspace{6mu}
    \begin{bmatrix}
      \mathrlap{\mspace{10mu}0}\hphantom
      {\;\;\tfrac{\partial}{\partial \eta}} &{\phantom{-}\tfrac{\partial}{\partial \eta}}
    \end{bmatrix}
    \begin{bmatrix}
      \bm{L}^{-1} & 0 \\
      0 & \bm{C}^{-1}
    \end{bmatrix}
    \begin{pmatrix}\bm{\psi}(t)\\\bm{q}(t)\end{pmatrix},\quad t\ge0.
  \end{aligned}
\end{equation}
where the parameter functions $\bm{C}, \bm{L}, \bm{G}, \bm{R}\colon [0,1]\to\C^{k\times k}$ stand for transverse capacitance, longitudinal inductance, transverse conductance, and longitudinal resistance, resp. 
The voltages and currents along the transmission line at time $t$ are given by the functions $\bm{V}(t),\bm{I}(t)\colon [0,1]\to\R^k$ with 
\begin{equation}\label{eq:input1}
\bm{V}(t) \coloneqq \bm{C}^{-1}\bm{q}(t),\quad    
\bm{I}(t) \coloneqq \bm{L}^{-1}\bm{\psi}(t)+\bm{I}_{\ext}(t).
\end{equation}    

\noindent 
Typical assumptions concerning the parameters are those presented in \cite{JaZw12}.

\begin{assumption}[Transmission lines - parameters]\label{ass:tl}
  $k\in\N$, and $\bm{C},\bm{L},\bm{R},\bm{G}:[0,1]\to\C^{k\times k}$ are measurable and essentially bounded. Moreover, 
\begin{equation*}
\bm{C}(\eta) > 0,\quad
\bm{L}(\eta) > 0,\quad
\bm{R}(\eta) + \bm{R}(\eta)\adjun \geq 0,\quad \text{and} \quad
\bm{G}(\eta) + \bm{G}(\eta)\adjun \geq 0
\end{equation*}
for almost every $\eta \in [0,1]$,
and the pointwise inverses 
  $\bm{C}^{-1},\bm{L}^{-1}\colon [0,1]\to\C^{k\times k}$ are essentially bounded as well.
\myendhere
\end{assumption}

\begin{remark}\label{rem:tl}
As the functions involved in \eqref{eq:teleq} are $\C^k$-valued, this model effectively represents $k$ transmission lines. The off-diagonal components of $\bm{C}$, $\bm{L}$, $\bm{R}$, and $\bm{G}$ account for effects like cross-talk and cross-losses. Since the main focus of this article is the analysis and modeling of the interaction between transmission lines and the electromagnetic field, these effects are essentially embedded in the electromagnetic radiation. Therefore, it is sufficient to assume that $\bm{C}$, $\bm{L}$, $\bm{R}$, and $\bm{G}$ are pointwise diagonal matrices. However, making this additional assumption does not lead to a simplification in terms of mathematical complexity, so we do not require that $\bm{C}$, $\bm{L}$, $\bm{R}$, and $\bm{G}$ be pointwise diagonal.
\end{remark}

Before turning to the initial and boundary values of the transmission lines, we take a step back to justify the transmission line model \eqref{eq:teleq} from first principles. As mentioned in the introduction, all electrical and magnetic phenomena are governed by Maxwell's equations. We now use these equations as a starting point and derive the transmission line model through a series of simplifying assumptions. This derivation follows the lines of~\cite{bagu18}. It is motivated by the fact that various modeling approaches exist for the coupling terms $\bm{I}_{\ext}$ and $\bm{E}_{\ext}$; see \Cref{rem:alttlmodel}.

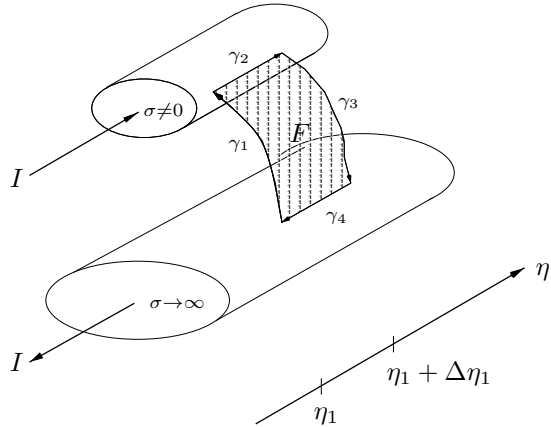
\begin{figure}[bt]
\begin{center}
\setlength{\unitlength}{.8cm}
\begin{picture}(24.0,8.0)
  \put(3.5,0.2){{\includegraphics[height=5.6cm]{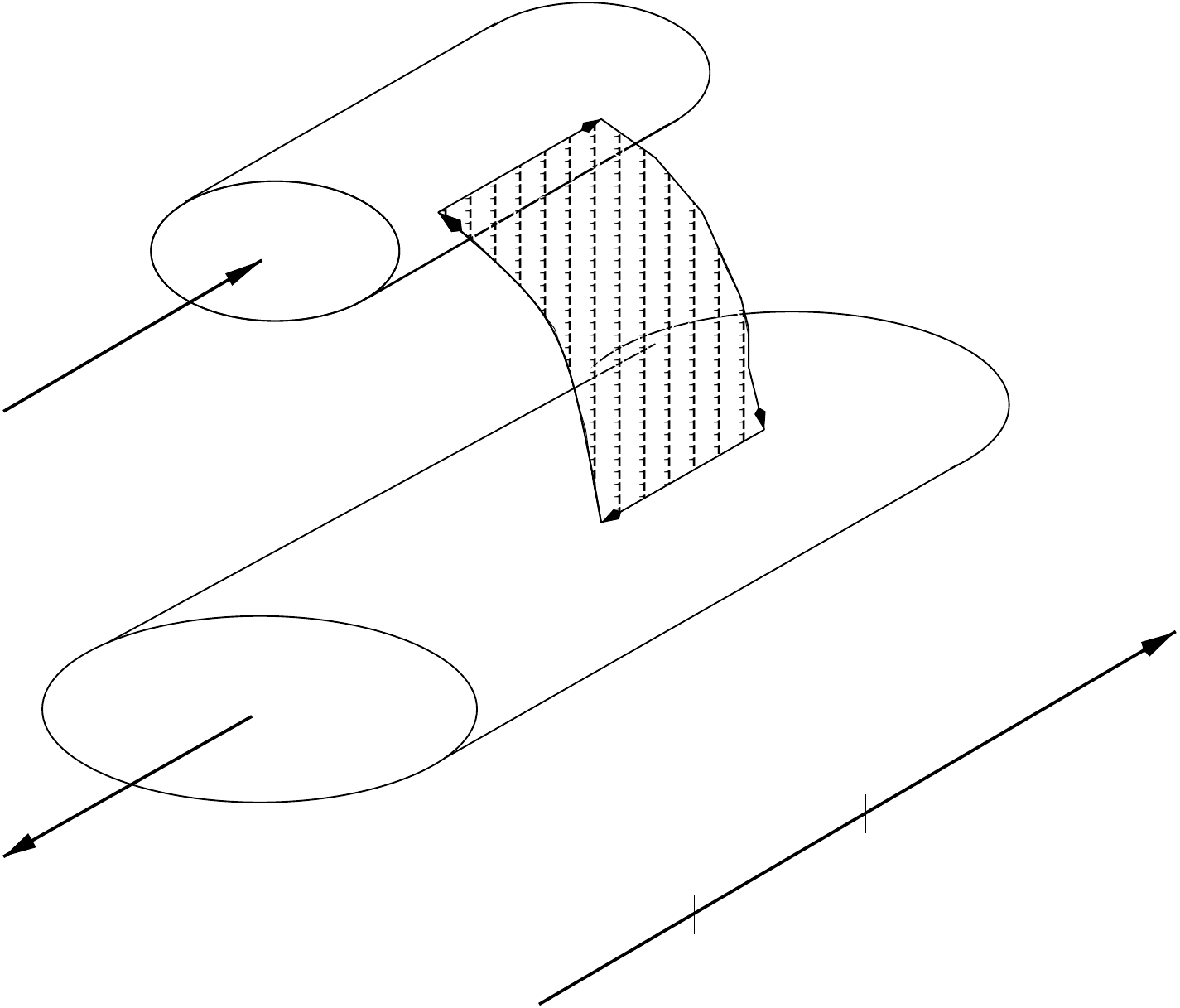}}}
\put(11.85,2.7){$\eta$}
\put(5.5,2.14){$\scriptstyle \sigma \rightarrow \infty$}
\put(5.4,5.32){$\scriptstyle\sigma \neq 0$}
\put(3.2,4.15){$I$}
\put(3.2,1.1){$I$}
\put(8.2,0.3){$\eta_1$}
\put(9.4,0.95){$\eta_1+\Delta \eta_1$}
\put(7.8,4.9){$F$}
\put(6.8,4.8){$\scriptstyle \gamma_1$}
\put(6.8,6.25){$\scriptstyle \gamma_2$}
\put(8.6,5.5){$\scriptstyle \gamma_3$}
\put(8.4,3.6){$\scriptstyle \gamma_4$}
\end{picture}
\end{center}
\caption{Sketch of a lossy transmission line with supply ($\sigma \neq 0$) 
         and return ($\sigma \rightarrow \infty$) conductor (taken from~\cite{bagu18}) 
        \label{fig.tml}}
\label{lbahn2}
\end{figure}%

\Cref{fig.tml} illustrates a lossy transmission line with a supply and a return conductor.  For our notation, we align the $\eta$-axis with
the transmission line and the return conductor. 
The regular structure
of transmission lines allows for some simplifying assumptions, which lead to
a one-dimensional model~\cite{grab91}:
\begin{enumerate}
  \item \emph{No skin effect}. Disregarding the skin effect of a lossy conductor, the current density is assumed to be homogeneous across any cross-section perpendicular to the direction of signal propagation. Accordingly, the line current $\bm{I}(t,\eta)$ at position $\eta$ of the cable is given by the oriented surface integral
  \[
    \bm{I}(t,\eta) \coloneqq
    \int_{S(\eta)} \bm{J}(t,\xi) \cdot \dx[\bm{S}(\xi)] 
  \]
  where $S(\eta)$ denotes the cross-sectional area at $\eta$, and $\bm{J}$ is the current density. We define analogously
  \[
    \bm{I}_{\ext}(t,\eta) \coloneqq
    \int_{S(\eta)} \bm{J}_{\ext}(t,\xi) \cdot \dx[\bm{S}(\xi)]
  \]
  \item \emph{Quasi-stationary behavior transversal to the direction of propagation}. Assuming such a behavior, the magnetic field component in the propagation direction is constant in time.  Hence, for a path $\gamma$ which connects the conductor’s surface with the return conductor and which is completely in a plane perpendicular to the $\eta$-axis, the value of the potential $\bm{V}$ (the line voltage) defined as the line integral
  \begin{equation}
    \bm{V}(t,\eta) \coloneqq \int_{\gamma} \bm{E}(t,\xi) \cdot \dx[\bm{s}(\xi)]
  \end{equation}
  is independent of the actual path $\gamma$.
  
  \item \emph{Linear materials}. Assuming linear materials (and quasi-stationarity), the charge density (per unit length) $\bm{q}$ is proportional to the line voltage $\bm{V}$
  \begin{equation}
    \bm{q}(t,\eta) = \bm{C}(\eta)\cdot \bm{V}(t,\eta) \quad\text{with}\quad
    \bm{q}(t,\eta) \coloneqq \int_{S(\eta)} \bm{\rho}(t,\xi) \cdot \dx[\xi],
  \end{equation}
where the latter is an integral with respect to the two-dimensional surface measure.
Further, the flux density (per unit length) $\bm{\psi}$ is proportional to the line current $\bm{I}$,
  \begin{equation}
    \bm{\psi}(t,\eta) = \bm{L}(\eta)\cdot \bm{I}(t,\eta)
    \quad\text{with}\quad
    \bm{\psi}(t,\eta) \coloneqq \lim_{\Delta \eta \rightarrow 0} \frac{1}{\Delta \eta} \int_{F} \bm{B}(t,\xi) \cdot \dx[\bm{S}(\xi)],
  \end{equation}
where the surface $F$ is as in \Cref{lbahn2}. Further, the conductivity $\bm{\sigma}$ is constant on any perpendicular cross-section. That is, there exists some $\sigma\colon [0,1]\to\R$ with
\begin{equation}
\bm{\sigma}(\xi)=\sigma(\eta)\quad \text{ for all }\xi\in S(\eta).\label{eq:sigmaconst}
\end{equation}
\end{enumerate}
Based on these assumptions, a transmission line model can be derived from Maxwell's equations: 
\begin{itemize}
  \item Maxwell’s first law yields by integration (see \Cref{fig.tml})
  \begin{equation}
  \label{eq.tml1}
    \oint_{\partial F} \bm{E}(t,\xi) \cdot \dx[\bm{s}(\xi)] = - \int_{F} \frac{\partial}{\partial t} \bm{B}(t,\xi) \cdot \dx[\bm{S}(\xi)].
  \end{equation}
Using the second assumption and the constitutive relation $\bm{J}(t,\xi) = \bm{\sigma}(\xi) \cdot \bm{E}(t,\xi)$, together with \eqref{eq:sigmaconst}, and denoting the surface measure of $S(\eta)$ by $|S(\eta)|$, we obtain for the terms on the left-hand side that
  \begin{align*}
    \int_{\gamma_1} \bm{E}(t,\xi) \cdot \dx[\bm{s}(\xi)]
    &= -\bm{C}(\eta_1)^{-1} \bm{q}(t,\eta_1), \\
    \int_{\gamma_2} \bm{E}(t,\xi) \cdot \dx[\bm{s}(\xi)]
    &= \int_{\eta_1}^{\eta_1+\Delta \eta} \frac{1}
    {\sigma(\eta) |S(\eta)|} \bm{L}(\eta)^{-1} \bm{\psi}(t,\eta) \cdot \dx[\eta], \\
    \int_{\gamma_3} \bm{E}(t,\xi) \cdot \dx[\bm{s}(\xi)]
    &= \bm{C}(\eta)^{-1} \cdot \bm{q}(t,\eta_1+\Delta \eta), \\
    \int_{\gamma_4} \bm{E}(t,\xi) \cdot \dx[\bm{s}(\xi)]
    &= 0.
  \end{align*}
Multiplying from left with $1/\Delta \eta$,   taking the limit $\Delta \eta\rightarrow 0$, using the third assumption, and rearranging all terms of the left hand-side, \eqref{eq.tml1} becomes
  \begin{equation}
    \label{eq.tml1.new}
    \frac{\partial}{\partial \eta} \left( \bm{C}(\eta)^{-1} \cdot \bm{q}(t,\eta) \right)  + \bm{R}(\eta)\bm{L}(\eta)^{-1} \bm{\psi}(t,\eta) + \frac{\partial}{\partial t} \bm{\psi}(t,\eta) =0,
  \end{equation}
  where we have defined $\bm{R}(\eta) \coloneqq 1/(\bm{\sigma} \cdot \abs{S(\eta)})$ as the longitudinal resistance.
  \item Integrating the charge conservation $\frac{\partial}{\partial t} \rho + \div \bm{J} + \div \bm{J}_{\ext} =0$ from $\eta_1$ to $\eta_1+\Delta \eta$,
  \begin{multline*}
    \int_{\eta_1}^{\eta_1+\Delta \eta} \frac{\partial}{\partial t} \bm{q}(t,\eta) \dx[\eta] + \big(\bm{I}(t,\eta_1+\Delta \eta) - \bm{I}(\eta_1)\big)
    \\
    + (\bm{I}_{\ext}(t,\eta_1 + \Delta \eta) - \bm{I}_{\ext}(\eta_1))
    = 0,
  \end{multline*}
  dividing by $\Delta \eta$, and taking the limit $\Delta \eta \rightarrow 0$, we deduce that
  \begin{equation}
    \label{eq.tml2.new}
    \frac{\partial}{\partial t}\bm{q}(t,\eta) + \frac{\partial}{\partial \eta} \left(\bm{L}^{-1} \bm{\psi}(t,\eta)+ \bm{I}_{\ext}(t,\eta) \right) =0.
  \end{equation}
\end{itemize}

If we finally add to the left-hand side of the latter equation the term $\bm{G}(\eta)\bm{C}(\eta)^{-1} \bm{q}(t,\eta)$ to include losses in the dielectric between the supply and return conductor that have not yet been taken into account, then equations~\eqref{eq.tml1.new} and \eqref{eq.tml2.new} are equivalent to the transmission line model~\eqref{eq:teleq} for a single transmission line.

\begin{remark}\label{rem:alttlmodel}
For a slightly different setup, \textsc{Agrawal} introduces in~\cite{Agrawal1980} a transmission line model with electric field excitation, namely
\[
\begin{aligned}
  \ddt
  \begin{pmatrix}
    \bm{\psi}(t) \\ \bm{q}(t)
  \end{pmatrix}
  &=
  \begin{bmatrix}
    -\bm{R} & -\tfrac{\partial}{\partial \eta} \\
    -\tfrac{\partial}{\partial \eta} & -\bm{G}
  \end{bmatrix}
  \begin{bmatrix}
    \bm{L}^{-1} & 0 \\
    0 & \bm{C}^{-1}
  \end{bmatrix}
  \begin{pmatrix}
    \bm{\psi}(t) \\
    \bm{q}(t)
  \end{pmatrix}
  +
  \begin{bmatrix}
    \id \\ 0
  \end{bmatrix}
  \bm{E}_{\ext}(t), \\
  \bm{I}_{\ext}(t)
  &=
  \begin{bmatrix}
    \id & 0
  \end{bmatrix}
  \begin{bmatrix}
    \bm{L}^{-1} & 0 \\
    0 & \bm{C}^{-1}
  \end{bmatrix}
  \begin{pmatrix}
    \bm{\psi}(t) \\
    \bm{q}(t)
  \end{pmatrix}, \quad t \geq 0,
\end{aligned}
\]
see also~\cite{Rach08}. We note that this leads to a different model than the one considered here. We have chosen the model~\eqref{eq:teleq} because it can be directly derived from Maxwell's equations, as outlined prior to this remark.
\end{remark}

\subsection{Initial and boundary conditions} 

The system is equipped with initial conditions $\bm{q}(0) = \bm{q}_0$, $\bm{\psi}(0) = \bm{\psi}_0$, where $\bm{q}_0, \bm{\psi}_0 \colon [0,1] \to \C^k$ are given. No further remarks are necessary.

In contrast, the boundary conditions for the voltage and current along the transmission line, as defined in \eqref{eq:input1}, require a more thorough discussion. For some $m \leq 2k$, with $W_{B,{\rm inp}} \in \C^{m\times 4k}$ and $W_{B,0} \in \C^{(2k-m)\times 4k}$, these are defined as follows:
\begin{equation}
u(t) = W_{B,{\rm inp}}
  \begin{pmatrix}
    \phantom{-}\bm{V}(t,0) \\
    \phantom{-}\bm{V}(t,1) \\
    \phantom{-}\bm{I}(t,0) \\
    -\bm{I}(t,1)
  \end{pmatrix},
\quad
  0 = W_{B,0}
  \begin{pmatrix}
    \phantom{-}\bm{V}(t,0) \\
    \phantom{-}\bm{V}(t,1) \\
    \phantom{-}\bm{I}(t,0) \\
    -\bm{I}(t,1)
  \end{pmatrix}.\label{eq:input2}
\end{equation}
Here, $u\colon \R_{\geq 0} \to \C^{m}$ represents the input of the system.
The negative sign in the current at $\eta=1$ originates from the fact that the current and voltage, unlike at the location $\eta=0$, are in opposite directions.

\begin{assumption}\label{ass:bndcont}
  The matrix $W_B\coloneqq [W_{B,{\rm inp}}^*,\,W_{B,0}^*]^*\in\C^{2k\times 4k}$ has full row rank, and
  \begin{equation}
    W_B
    \begin{bmatrix}
      0 & \id_{{2k}} \\
      \id_{{2k}} & 0
    \end{bmatrix}
    W_B^*\ge0.\label{eq:WBdef}
  \end{equation}
\end{assumption}
\begin{remark}\label{rem:bndcond}\
  \begin{enumerate}[label=(\alph{*})]
    \item\label{rem:bndconda} 

    The conditions in \Cref{ass:bndcont} encompass a crucial input configuration, in which each of the $k$ transmission lines is equipped with two (possibly homogeneous) boundary conditions. These boundary conditions involve either specifying both voltages, both currents, or the voltage on one side and the current on the other side.
    To be more precise in mathematical terms, this type of boundary condition can be expressed as follows: Denote 
        \begin{equation*}
      \tilde{u}(t)=\begin{pmatrix}\tilde{u}_{1}(t)\\\vdots\\ \tilde{u}_{k}(t)\end{pmatrix},\quad \tilde{u}_{i}(t)=\begin{pmatrix}\tilde{u}_{i,0}(t)\\\tilde{u}_{i,1}(t)\end{pmatrix},\; i=1,\ldots,k,
    \end{equation*}
    as the vector of boundary values corresponding to the input and the homogeneous boundary values. By denoting $\bm{V}_i$ and $\bm{I}_i$ as the $i$\textsuperscript{th} components of $\bm{I}$ and $\bm{V}$ in \eqref{eq:input1}, resp., the above described boundary condition mean that, for $i=1,\ldots,k$, 
    \begin{equation}
      \begin{aligned}
        &\left(\tilde{u}_{i,0}(t)=\bm{V}_i(t,0) \text{ or } \tilde{u}_{i,0}(t)=\phantom{-}\bm{I}_i(t,0)\right),\\
        \text{and }&\left(\tilde{u}_{i,1}(t)= \bm{V}_i(t,1)\text{ or } \tilde{u}_{i,1}(t)=-\bm{I}_i(t,1)\right).
      \end{aligned}\label{eq:bndcond}
    \end{equation}
Given that certain boundary conditions might be zero, the above types of boundary conditions contain the scenario where an $m$-dimensional input $u$ is given by
\[u(t)=E\tilde{u}(t),\]
where $E\in\R^{2k\times m}$ is a matrix with columns representing linearly independent canonical unit vectors. 
  A practical interpretation of these boundary conditions involves placing voltage and current sources at the ends of the transmission line, as illustrated in \Cref{Fig:TML_source}.
\begin{figure}[htbp]
\centering
\resizebox{\textwidth}{!}
{
\begin{circuitikz}[scale=0.7]

\draw  (0,10) node[ground]{} 
to[V=$\tilde{u}_{i{,}0}(t)$,*-] (0,12);
\draw  (20,10) node[ground]{} 
to[V_=$\tilde{u}_{i{,}1}(t)$,*-] (20,12);
\draw (-0,12) -- (0.5,12) to[TL] (19.5,12) -- (20,12);

\draw  (0,6) node[ground]{} 
to[V=$\tilde{u}_{i{,}0}(t)$,*-] (0,8);
\draw  (20,6) node[ground]{} 
to[I_=$\tilde{u}_{i{,}1}(t)$,*-] (20,8);
\draw (-0,8) -- (0.5,8) to[TL] (19.5,8) -- (20,8);

\draw  (0,2) node[ground]{} 
to[I=$\tilde{u}_{i{,}0}(t)$,*-] (0,4);
\draw  (20,2) node[ground]{} 
to[V_=$\tilde{u}_{i{,}1}(t)$,*-] (20,4);
\draw (-0,4) -- (0.5,4) to[TL] (19.5,4) -- (20,4);

\draw  (0,-2) node[ground]{} 
to[I=$\tilde{u}_{i{,}0}(t)$,*-] (0,0);
\draw  (20,-2) node[ground]{} 
to[I_=$\tilde{u}_{i{,}1}(t)$,*-] (20,0);
\draw (-0,0) -- (0.5,0) to[TL] (19.5,0) -- (20,0);

\end{circuitikz}
}
\caption{Boundary conditions for the transmission line}
\label{Fig:TML_source}
\end{figure}

  \item A ring-shaped cable can be modelled by employing the boundary conditions
  \begin{equation*}
    \bm{V}_i(t,0)=\bm{V}_i(t,1),\quad
    \bm{I}_i(t,0)=-\bm{I}_i(t,1).
  \end{equation*}
  
  \item\label{rem:bndcondb} The input configurations outlined above share the characteristic that the matrix on the left-hand side of \eqref{eq:WBdef} equals zero. ``True'' negative semi-definiteness can be achieved by boundary conditions as in \ref{rem:bndconda} by respectively replacing the boundary conditions in \eqref{eq:bndcond} with at least one of
  \begin{align}
    \tilde{u}_{i,0}(t)&=\phantom{-}\bm{V}_i(t,0)+\bm{R}_{{\rm ext},i0}\bm{I}_i(t,0),\label{eq:inp1}\\ \tilde{u}_{i,0}(t)&=\phantom{-}\bm{I}_i(t,0)+\bm{R}_{{\rm ext},i0}^{-1}
    \bm{V}_i(t,0),\label{eq:inp2}\\
    \tilde{u}_{i,1}(t)&=\phantom{-} \bm{V}_i(t,1)-\bm{R}_{{\rm ext},i1}\bm{I}_i(t,1),\label{eq:inp3}\\\tilde{u}_{i,1}(t)&=-\bm{I}_i(t,1)+\bm{R}_{{\rm ext},i1}^{-1}\bm{V}_i(t,1).\label{eq:inp4}
  \end{align}
  where $\bm{R}_{{\rm ext},i0}$, $\bm{R}_{{\rm ext},i1}$ are positive constants. The practical interpretation of these boundary conditions entails prescribing either the voltage or current in either a serial or parallel connection of a linear resistance with the transmission line, resp. We illustrate this in \Cref{Fig:TML_souce2}, where, for brevity, we display only such boundary conditions on the left-hand side.   
\begin{figure}[htbp]
\centering
\resizebox{\textwidth}{!}
{
\begin{circuitikz}[scale=0.7]

\draw  (0,10) node[ground]{} 
to[V=$\tilde{u}_{i{,}0}(t)$,*-] (0,12);
\draw (0,12) 
to[R=$\bm{R}_{{\rm ext}i0}$,-] (3,12)
 to[TL] (19.5,12) -- (20,12);

\draw  (0,5) node[ground]{} 
to[I=$\tilde{u}_{i{,}0}(t)$,*-] (0,8);
\draw (0,8) 
-- (3,8)
 to[TL] (19.5,8) -- (20,8);
\draw (1.5,8) 
to[R=$\bm{R}_{{\rm ext}i0}$,-*] (1.5,5) node[ground]{};

\end{circuitikz}
}
\caption{Boundary conditions for the transmission line}
\label{Fig:TML_souce2}
\end{figure}
  \end{enumerate}
\end{remark}

\subsection{Outputs} 

Our system is moreover equipped with an output $y\colon \R_{\ge0}\to\C^{p}$ of the form
\begin{equation}\label{eq:output}
  y(t) = W_{C,{\rm out}}
  \begin{pmatrix}
    \phantom{-}\bm{V}(t,0) \\
    \phantom{-}\bm{V}(t,1) \\
    \phantom{-}\bm{I}(t,0) \\
    -\bm{I}(t,1)
  \end{pmatrix},
\end{equation}
for some $W_{C,{\rm out}}\in\C^{p\times 4m}$. Our transmission line model can now be considered in a systems-theoretic manner as a system with each two types of inputs and outputs: the input itself along with the spatial current distribution, and the output itself along with the external distributed voltages. This is illustrated in the form of a block diagram in \Cref{fig:transmission-line-with-ports}.
\begin{figure}[htbp]
\centering
\begin{tikzpicture}
   \tikzstyle{mynode1} = [rectangle, minimum width=2.5cm, minimum height=1cm, text centered, draw=black]

   \coordinate (vh) at (0,-2); 
   \coordinate (vv) at (2,0); 

   \node[mynode1] (H3) at (0,0) {\parbox[c]{2cm}{\centering Transmission\\line}};

   \draw[-Latex] ($(H3|- H3.175) - (vv) + (-1,0)$) -- node[above] {$u$} (H3.175);
   \draw[Latex-] ($(H3|- H3.5) + (vv) + (1,0)$) -- node[above] {$y$}  (H3.5);

   \draw[-Latex] ($(H3|- H3.185) - (vv) + (-1,0)$) -- node[below] {$\bm{I}_{\rm ext}$} (H3.185);
   \draw[Latex-] ($(H3|- H3.-5) + (vv) + (1,0)$) -- node[below] {$\bm{E}_{\rm ext}$}  (H3.-5);

\end{tikzpicture}
\caption{\label{fig:transmission-line-with-ports}Transmission line as block diagram}
\end{figure}
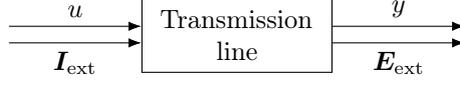

A special role is played by so-called \emph{co-located outputs}, which are defined in the sequel.
\begin{definition}\label{def:coloc}
  Assume that $W_{B,{\rm inp}}\in\C^{m\times 4k}$, $W_{B,0}\in\C^{(2k-m)\times 4k}$, $m\leq 2k$, $W_B \coloneqq [W_{B,{\rm inp}}^*,\,W_{B,0}^*]^*\in\C^{2m\times 4m}$ fulfill \Cref{ass:bndcont}.
  Then an output \eqref{eq:output} is called \em{co-located} to $u$ as in \eqref{eq:input1}, \eqref{eq:input2}, if $W_{C,{\rm out}}\in\C^{m\times 4k}$ has the form
  \begin{equation}\label{eq:col1}
    W_{C,{\rm out}}=[\id_m,\,0_{m\times (2k-m)}]\,W_C
  \end{equation}
  for some $W_C\in\C^{2k\times 4k}$ with the property that $[W_B^*,\,W_C^*]\in\C^{4k\times 4k}$ 
  with
  \begin{equation}
    \begin{bsmallmatrix}0&\id_{2k}\\\id_{2k}&0\end{bsmallmatrix}-\begin{bsmallmatrix}W_{B}\\W_C\end{bsmallmatrix}\adjun\begin{bsmallmatrix}0&\id_{2k}\\\id_{2k}&0\end{bsmallmatrix}\begin{bsmallmatrix}W_{B}\\W_C\end{bsmallmatrix}\leq0.\label{eq:col2}
  \end{equation}
\end{definition}

We will later observe that co-located outputs result in a system that establishes a~power balance, where the inner product of input and output can be interpreted as the power supplied to the system.
Next, we show the existence of co-located outputs. 
\begin{proposition}
Assume that $W_{B,{\rm inp}}\in\C^{m\times 4k}$, $W_{B,0}\in\C^{(2k-m)\times 4k}$, $m\leq 2k$, $W_B \coloneqq [W_{B,{\rm inp}}^*,\,W_{B,0}^*]^*\in\C^{2k\times 4k}$ fulfill \Cref{ass:bndcont}. Then there exists some $W_{C,{\rm out}}\in\C^{m\times 4k}$, such that 
\eqref{eq:col1} and \eqref{eq:col2} holds for some $W_{C}\in\C^{2k\times 4k}$.
\end{proposition}

\begin{proof}
It suffices to show that there exists some 
$W_{C}\in\C^{2k\times 4k}$ such that  \eqref{eq:col2} holds.\\
Partition $W_{B}=[W_{B1},W_{B2}]$ for $W_{B1},W_{B2}\in\C^{2k\times 2k}$, and 
let $U,V\in\C^{2k\times 2k}$ be unitary matrices, such that, for $r\coloneq \rk W_{B1}$ and some $W_{B111}\in\C^{r\times r}$,
\[U^*W_{B1}V=\left[\begin{smallmatrix}
    W_{B111}&0\\0&0
\end{smallmatrix}\right].\]
Further partitioning
\[U^*W_{B2}V=\left[\begin{smallmatrix}
W_{B211}&W_{B212}\\W_{B221}&W_{B222}
\end{smallmatrix}\right]\]
according to the previous block structure, we obtain from
\begin{multline*}
    0\leq W_{B2}W_{B1}^*+W_{B1}W_{B2}^*\\=
   U\left(\left[\begin{smallmatrix} W_{B211}&W_{B212}^*\\W_{B221}&W_{B222}
\end{smallmatrix}\right]\left[\begin{smallmatrix}
    W_{B111}^*&0\\0&0
\end{smallmatrix}\right]+
\left[\begin{smallmatrix}
    W_{B111}&0\\0&0
\end{smallmatrix}\right]
\left[\begin{smallmatrix}
    W_{B211}^*&W_{B221}^*\\W_{B212}^*&W_{B222}^*
\end{smallmatrix}\right]
\left[\begin{smallmatrix}
    W_{B111}&0\\0&0
\end{smallmatrix}\right]\right)U^*\\
=U\left[\begin{smallmatrix}    W_{B111}W_{B211}^*+W_{B211}W_{B111}^*&W_{B111}W_{B221}^*\\W_{B221}W_{B111}^*&0
\end{smallmatrix}\right]U^*.
\end{multline*}
Now the invertibility of $W_{B111}$ leads to $W_{B221}=0$. The assumption that $W_B$ has full row rank then implies that $W_{B222}$ is invertible. Now we define $W_C=[W_{C1},\,W_{C2}]$ with
\[W_{C1}=U \left[\begin{smallmatrix}    0&0\\0&W_{B222}^{-*}
\end{smallmatrix}\right]V^*,\; W_{C2}=U\left[\begin{smallmatrix}    W_{B111}^{-*}&0\\-W_{B222}^{-*}W_{B212}^{*}W_{B211}^{-*}&0
\end{smallmatrix}\right]V^*.\]
Then
\begin{align*}
\MoveEqLeft
W_{C2}W_{C1}^*+W_{C1}W_{C2}^*\\
&= U\left(\left[\begin{smallmatrix}    W_{B111}^{-*}&0\\-W_{B222}^{-*}W_{B212}^{*}W_{B211}^{-*}&0
\end{smallmatrix}\right]\left[\begin{smallmatrix}    0&0\\0&W_{B222}^{-1}
\end{smallmatrix}\right]\right.\\&\qquad+\left.\left[\begin{smallmatrix}    0&0\\0&W_{B222}^{-*}
\end{smallmatrix}\right]\left[\begin{smallmatrix}    W_{B111}^{-1}&-W_{B211}^{-1}W_{B212}W_{B222}^{-1}\\0&0
\end{smallmatrix}\right]\right)U^*
=U\left[\begin{smallmatrix}0&0\\0&0
\end{smallmatrix}\right]U^*=0, \\[1ex]
\MoveEqLeft
W_{B2}W_{C1}^*+W_{B1}W_{C2}^*\\
&= U\left(\left[\begin{smallmatrix}    W_{B211}&W_{B212}\\0&W_{B222}
\end{smallmatrix}\right]\left[\begin{smallmatrix}    0&0\\0&W_{B222}^{-1}
\end{smallmatrix}\right]+\left[\begin{smallmatrix}    W_{B111}&0\\0&0
\end{smallmatrix}\right]\left[\begin{smallmatrix}    W_{B111}^{-1}&-W_{B211}^{-1}W_{B212}W_{B222}^{-1}\\0&0
\end{smallmatrix}\right]\right)U^*\\
&= U\left[\begin{smallmatrix}\id_r&0\\0&\id_{2k-r}
\end{smallmatrix}\right]U^*=\id_{2k},
\end{align*}
and we obtain that
\[
\begin{bmatrix}W_B\\{W}_C\end{bmatrix}\begin{bmatrix}0&\id_{2k}\\\id_{2k}&0\end{bmatrix}\begin{bmatrix} W_B^*&{W}_C^*\end{bmatrix}
=\begin{bmatrix}W_{B1}W_{B2}^*+W_{B2}W_{B1}^*&\id_{2k}\\
\id_{2k}&0\end{bmatrix}.    
\]
Then the desired result follows from
\[
  W_{B1}W_{B2}^*+W_{B2}W_{B1}^*= W_B
  \begin{bmatrix}
    0 & \id_{{2k}} \\
    \id_{{2k}} & 0
  \end{bmatrix}
  W_B^*\ge 0. \qedhere
\]
\end{proof}

The latter result is of a rather abstract nature. Next we provide detailed discussion of the practical interpretation of co-located outputs.
\begin{remark}\label{rem:iocol}
  Considering the boundary conditions described in \Cref{rem:bndcond}\,\ref{rem:bndconda}, a co-located output is represented by the currents at the locations where the input is a voltage, and the voltage at the locations where the input is the current. In the notation of \Cref{rem:bndcond}, this means
  \begin{equation*}
    y(t) = E^*\tilde{y}(t),\quad \tilde{y}(t) = \begin{pmatrix}\tilde{y}_{1}(t) \\ \vdots\\ \tilde{y}_{k}(t)\end{pmatrix},\quad \tilde{y}_{i}(t) = \begin{pmatrix}\tilde{y}_{i,0}(t) \\ \tilde{y}_{i,1}(t)\end{pmatrix},\; i=1,\ldots,k,
  \end{equation*}
  where
  \begin{alignat*}{3}
    \tilde{y}_{i,0}(t)\;&=\phantom{-}\bm{I}_i(t,0),\qquad&\text{if}&\qquad&\tilde{y}_{i,0}(t)\;&=\phantom{-}\bm{V}_i(t,0),\\
    \tilde{y}_{i,0}(t)\;&=\phantom{-}\bm{V}_i(t,0),&\text{if}&&\tilde{y}_{i,0}(t)\;
    &=\phantom{-}\bm{I}_i(t,0),\\
    \tilde{y}_{i,1}(t)\;&=-\bm{I}_i(t,1),&\text{if}&&\tilde{y}_{i,1}(t)\;&=\phantom{-}\bm{V}_i(t,1),\\
    \tilde{y}_{i,1}(t)\;&=\phantom{-}\bm{V}_i(t,1),&\text{if}&&\tilde{y}_{i,1}(t)\;&=-\bm{I}_i(t,1),\\
  \end{alignat*}
  In situations where the input consists of voltages or currents connected in series or parallel with the transmission line and a resistance (as discussed in \Cref{rem:bndcond},\ref{rem:bndconda}), co-located outputs can be determined as follows (for brevity, we consider only the left-hand side of the transmission line):
  If \eqref{eq:inp1} is applicable, co-located outputs are represented as $\tilde{y}_{i,0}(t)=\bm{I}_i(t,0)$.
  Alternatively, when \eqref{eq:inp2} is in use, co-located outputs can be expressed as $\tilde{y}_{i,0}(t)=\bm{V}_i(t,0)$.
\end{remark}

\subsection{Power balance}\label{sec:enbaltl}

We now consider, at least formally, the energy evolution in the system along solutions of the transmission line model. Under \Cref{ass:tl} on the system parameters and \Cref{ass:bndcont} on the boundary conditions (i.e., the structure of the input $u$), we additionally assume that the output $y$ is co-located in the sense of \Cref{def:coloc}.

Given are the magnetic flux $\bm{\psi}(t)$ and electric charge $\bm{q}(t)$ at time $t$. While the magnetic energy density at $\eta \in [0,1]$ is $\tfrac{1}{2} \bm{\psi}(t,\eta)\adjun \bm{L}(\eta)^{-1} \bm{\psi}(t,\eta)$, the electric energy density is $\tfrac{1}{2} \bm{q}(t,\eta)\adjun \bm{C}(\eta)^{-1} \bm{q}(t,\eta)$. Consequently, the total energy of the transmission line at time $t$ is given by the sum of the spatial integrals of these energy densities over $[0,1]$, i.e.,
\[\mathcal{E}_{\rm tl}(\bm{\psi}(t),\bm{q}(t))=\frac12\int_{0}^1\bm{\psi}(t,\eta)\adjun \bm{L}(\eta)^{-1} \bm{\psi}(t,\eta)+\bm{q}(t,\eta)\adjun \bm{C}(\eta)^{-1} \bm{q}(t,\eta) \dx[\eta].\]
To analyze the power in the system, we differentiate the total energy with respect to time. 
Hereby, we use that the model equations \eqref{eq:teleq} together with \eqref{eq:input1} yield that
\begin{align*}
    \ddts \bm{\psi}(t)&=-\bm{R}\bm{L}^{-1}\bm{\psi}(t)-\tfrac{\partial}{\partial \eta}\bm{V}(t),
    \\    
    \ddts \bm{q}(t)&=-\bm{G}\bm{C}^{-1}\bm{q}(t)-\tfrac{\partial}{\partial \eta}\bm{I}(t).
\end{align*}
Now, using the fact that $\bm{R} + \bm{R}\adjun$ and $\bm{G} + \bm{G}\adjun$ are pointwise positive semi-definite (see \Cref{ass:tl}), and applying the product rule along with the fundamental theorem of calculus, we can estimate the power (i.e., the time derivative of the energy) by boundary terms, namely
\begin{align*}
  \MoveEqLeft
  \ddts\mathcal{E}_{\rm tl}(\bm{\psi}(t),\bm{q}(t))\\
  &=\Re\int_{0}^1\bm{\psi}(t,\eta)\adjun \bm{L}(\eta)^{-1} \ddts\bm{\psi}(t,\eta)+\bm{q}(t,\eta)\adjun \bm{C}(\eta)^{-1} \ddts\bm{q}(t,\eta) \dx[\eta]\\
  &= -\Re\int_{0}^1\bm{\psi}(t,\eta)\adjun \bm{L}(\eta)^{-1}\bm{R}(\eta)\bm{L}(\eta)^{-1} \bm{\psi}(t,\eta)\\&\qquad\qquad\qquad+\bm{q}(t,\eta)\adjun \bm{C}(\eta)^{-1} \bm{G}(\eta)\bm{C}(\eta)^{-1} \bm{q}(t,\eta) \dx[\eta]\\
  &\quad -\Re\int_{0}^1\bm{\psi}(t,\eta)\adjun \bm{L}(\eta)^{-1}\tfrac{\partial}{\partial \eta} \bm{V}(t,\eta)+\bm{q}(t,\eta)\adjun \bm{C}(\eta)^{-1}\tfrac{\partial}{\partial \eta} \bm{I}(t,\eta) \dx[\eta]\\
  &\leq  -\Re\int_{0}^1\bm{I}(t,\eta)\adjun\tfrac{\partial}{\partial \eta} \bm{V}(t,\eta)+\bm{V}(t,\eta)\adjun \tfrac{\partial}{\partial \eta} \bm{I}(t,\eta) \dx[\eta]\\
  &\quad+\Re \int_{0}^1\bm{I}_{\ext}(t,\eta)\adjun \tfrac{\partial}{\partial \eta}\bm{V}(t,\eta) \dx[\eta]\\
  &=  -\Re\left.\phantom{\int_0^1\!\!\!\!\!\!\!}\bm{I}(t,\eta)\adjun\bm{V}(t,\eta)\right|_{\eta=0}^{\eta=1}+\Re\int_{0}^1\bm{I}_{\ext}(t,\eta)\adjun \bm{E}_{\ext}(t,\eta) \dx[\eta].
\end{align*}
Further, we have
\begin{multline*}
-\Re\left.\phantom{\int_0^1\!\!\!\!\!\!\!\!\!}\big(\bm{I}(t,\eta)\adjun\bm{V}(t,\eta)\big)\right|_{\eta=0}^{\eta=1}=\frac12\begin{psmallmatrix}
    \phantom{-}\bm{V}(t,0) \\
    \phantom{-}\bm{V}(t,1) \\
    \phantom{-}\bm{I}(t,0) \\
    -\bm{I}(t,1)
  \end{psmallmatrix}\adjun \begin{bmatrix}
      0 & \id_{{2k}} \\
      \id_{{2k}} & 0
    \end{bmatrix}\begin{psmallmatrix}
    \phantom{-}\bm{V}(t,0) \\
    \phantom{-}\bm{V}(t,1) \\
    \phantom{-}\bm{I}(t,0) \\
    -\bm{I}(t,1)
  \end{psmallmatrix}   \\ \stackrel{\eqref{eq:col2}}{\leq}  
\frac12\begin{psmallmatrix}
    \phantom{-}\bm{V}(t,0) \\
    \phantom{-}\bm{V}(t,1) \\
    \phantom{-}\bm{I}(t,0) \\
    -\bm{I}(t,1)
  \end{psmallmatrix}\adjun  \begin{bmatrix}W_{B}\\W_C\end{bmatrix}\adjun\begin{bmatrix}0&\id_{2k}\\\id_{2k}&0\end{bmatrix}\begin{bmatrix}W_{B}\\W_C\end{bmatrix}\begin{psmallmatrix}
    \phantom{-}\bm{V}(t,0) \\
    \phantom{-}\bm{V}(t,1) \\
    \phantom{-}\bm{I}(t,0) \\
    -\bm{I}(t,1)
  \end{psmallmatrix} . 
\end{multline*}
The input relation \eqref{eq:input2} and the definition of the co-located output (see \Cref{def:coloc}) give 
\[\begin{pmatrix}
    u(t) \\
    0
  \end{pmatrix}=W_B\begin{psmallmatrix}
    \phantom{-}\bm{V}(t,0) \\
    \phantom{-}\bm{V}(t,1) \\
    \phantom{-}\bm{I}(t,0) \\
    -\bm{I}(t,1)
  \end{psmallmatrix},\qquad \begin{pmatrix}
    y(t) \\
    z(t)
  \end{pmatrix}=W_C\begin{psmallmatrix}
    \phantom{-}\bm{V}(t,0) \\
    \phantom{-}\bm{V}(t,1) \\
    \phantom{-}\bm{I}(t,0) \\
    -\bm{I}(t,1)
  \end{psmallmatrix}\]
for some $z(t)\in \C^{2k-m}$. The latter two relations give
\begin{multline*}
-\Re\left.\phantom{\int_0^1\!\!\!\!\!\!\!\!\!\!}\big(\bm{I}(t,\eta)\adjun\bm{V}(t,\eta)\big)\right|_{\eta=0}^{\eta=1}\leq \frac12 \begin{pmatrix}
   \begin{psmallmatrix} u(t)\\0\end{psmallmatrix}\\[1mm]\begin{psmallmatrix} y(t)\\z(t)\end{psmallmatrix}
  \end{pmatrix}\adjun\begin{bmatrix}0&\id_{2k}\\\id_{2k}&0\end{bmatrix}\begin{pmatrix}
   \begin{psmallmatrix} u(t)\\0\end{psmallmatrix}\\[1mm]\begin{psmallmatrix} y(t)\\z(t)\end{psmallmatrix}
  \end{pmatrix}\\=\Re \big(u(t)\adjun y(t)\big).
  \end{multline*}
Hence, the total power balance is given by
\begin{equation}
  \ddts\mathcal{E}_{\rm tl}(\bm{\psi}(t),\bm{q}(t))
\leq \Re \big(u(t)\adjun y(t)\big)+\Re\int_{0}^1\bm{I}_{\ext}(t,\eta)\adjun \bm{E}_{\ext}(t,\eta) \dx[\eta],\label{eq:tlenbal}
\end{equation}
The expression $\Re \big(u(t)\adjun y(t)\big)$ stands for the externally provided power, whereas 
\[\Re\int_{0}^1\bm{I}_{\ext}(t,\eta)\adjun \bm{E}_{\ext}(t,\eta) \dx[\eta]\]
is the power supplied by the external current and external field intensity. Two terms are responsible for the inequality: first,
\begin{multline}
\Re\int_{0}^1\bm{\psi}(t,\eta)\adjun \bm{L}(\eta)^{-1}\bm{R}(\eta)\bm{L}(\eta)^{-1} \bm{\psi}(t,\eta)\\+\bm{q}(t,\eta)\adjun \bm{C}(\eta)^{-1} \bm{G}(\eta)\bm{C}(\eta)^{-1} \bm{q}(t,\eta) \dx[\eta]\geq0\label{eq:distrtlloss}\end{multline}
represents the power dissipation within the transmission line. Second, the nonnegative term
\begin{equation}
    \label{eq:endtlloss}
\Re\big(\bm{I}(t,0)\adjun\bm{V}(t,0)\big)-\Re\big(\bm{I}(t,1)\adjun\bm{V}(t,1)\big)-\Re \big(u(t)\adjun y(t)\big)\end{equation}
describes the power dissipated at the ends of the transmission line. This occurs, for instance, due to the resistors described in \Cref{rem:bndcond}.

\section{The electromagnetic field}\label{sec:elmag}

Here, we provide a concise introduction to the fundamentals of electromagnetic field dynamics as governed by Maxwell's equations within a domain $\Omega \subset \R^3$. 
Let us first consolidate our assumptions regarding the domain and the associated parameters.
\begin{assumption}[Maxwell's equations - parameters]\label{ass:Maxwell}
$\bm{\epsilon},\bm{\mu},\bm{\sigma}\colon \Omega\to\C^{3\times 3}$ are measurable and essentially bounded. Moreover, 
\begin{multline*}
\bm{\epsilon}(\xi) > 0,\quad
\bm{\mu}(\xi) > 0,\quad \text{and} \quad
\bm{\sigma}(\xi) + \bm{\sigma}(\xi)\adjun \geq 0\\
\text{for almost every } \xi \in \Omega,
\end{multline*}
and the pointwise inverses 
  $\bm{\epsilon}^{-1},\bm{\mu}^{-1}\colon \Omega\to\C^{3\times 3}$ are essentially bounded as well.
    \myendhere
\end{assumption}

We begin by describing the properties and characteristics of the domain, followed by a presentation of the boundary conditions for Maxwell's equations used in our problem formulation.

\subsection{The spatial domain}\label{sec:Omega}

For $k$ being the number of transmission lines, we assume that the electromagnetic field evolves within a~domain $\Omega\subseteq\R^{3}$ structured as
\begin{equation}\label{eq:domain-Omega}
  \Omega = \Omega_{0} \setminus \bigcup_{i=1}^{k} \cl{\Omega_{i}},
\end{equation}
where $\Omega_{0} \subseteq \R^{3}$ Lipschitz domain, and the sets
$\Omega_{1},\ldots,\Omega_{k} \subseteq \R^{3}$  fulfill
\begin{align*}
  \cl{\Omega_{i}} \subseteq \Omega_{0},\quad &i=1,\ldots,k,\\
  \cl{\Omega_{i}} \cap \cl{\Omega_{j}} = \emptyset,\quad &i,j=1,\ldots,k \ \text{with}\  i\neq j
\end{align*}
The interpretation of the aforementioned domains is as follows: The entirety of the physical process occurs within $\Omega_{0}$. In essence, our field-cable interaction represents a physically closed procedure within $\Omega_{0}$. This domain is referred to as the ``computational domain''. It is worth noting that we do not impose any boundedness requirements on $\Omega_{0}$, making the choice of $\Omega_{0} = \R^{3}$ a viable option. The domains $\Omega_{1}, \ldots, \Omega_{k}$ correspond to the spatial regions of the cables, as depicted in \Cref{fig:domain}. We will assume that they possess a tubular shape, as explained in the following part.
\begin{figure}[htbp]
  \centering
  \def\svgwidth{0.7\textwidth}
  \iftoggle{birk}{\def\svgwidth{0.8\textwidth}}{}
  \iftoggle{default}{\def\svgwidth{1\textwidth}}{}
  \input{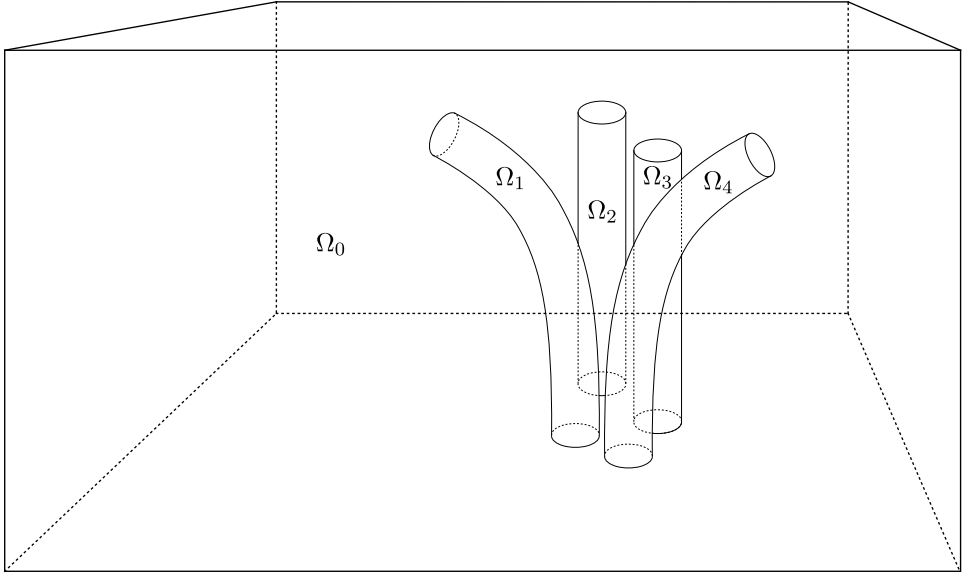}
  \caption{\label{fig:domain}Spatial domain with cables}
\end{figure}

The cables are allowed to be bent, and we assume that each cable has a~circle-shaped cross-sectional area of constant radius. Denote the radius of this cross-sectional area of the $i$\textsuperscript{th} cable by $r_i$, and let $l_{i}$ be its length. The center curve (i.e., the curve whose trace is consisting of the centers of the cross-sectional circles) is denoted by $\alpha_{i}\colon [0,1]\to\R^3$, see \Cref{fig:cable}. We assume the center curve to be twice continuously differentiable with constant infinitesimal arc length $l_{i}$,
and curvature is bounded by $\frac{1}{r_{i}}$, i.e.,
\begin{equation*}
  \norm{\alpha_{i}^{\prime}(\eta)} = l_{i}
  \quad\text{and}\quad
  \norm{\alpha''(\eta)} < \frac{l_{i}^{2}}{r_{i}} \quad\text{for all}\quad \eta \in [0,1].
\end{equation*}

\begin{figure}[htbp]
  \centering
  \def\svgwidth{0.35\textwidth}
  \input{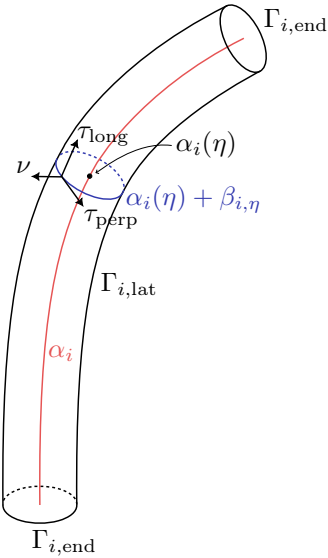}
  \smallskip
  \caption{\label{fig:cable}Parameterization of the cable}
\end{figure}
\noindent By using \Cref{le:normal-fields-of-path}, we obtain that there exist $\kappa_{i1},\kappa_{i2}\in \conC^{1}([0,1], \R^{3})$ , such that, for all $\eta\in[0,1]$,
\begin{equation*}
  (\tfrac{1}{l_i}\alpha_{i}^{\prime}(\eta),\kappa_{i1}(\eta),\kappa_{i2}(\eta))
\end{equation*}
is an orthonormal basis of $\R^{3}$ with
\begin{equation*}
  \det[\alpha_{i}^{\prime}(\eta),\,\kappa_{i1}(\eta),\,\kappa_{i2}(\eta)] = l_i.
\end{equation*}
\noindent
The lateral boundary $\Gamma_{i,\textup{lat}}$ of the $i$\textsuperscript{th} cable is now parameterised by
\begin{equation}\label{eq:Gammalat}
\begin{split}
&\Phi_i\colon\mapping{[0,1]\times\lparen -\uppi,\uppi \rbrack}{\R^3}
    {(\eta,\theta)}{\alpha_i(\eta)+\beta_{i\eta}(\theta)}
    \\[1.5ex]
  \mathllap{\text{with}\quad}
  &\beta_{i\eta}(\theta)
    = r_i\big(\kappa_{i1}(\eta)\sin(\theta) + \kappa_{i2}(\eta)\cos(\theta)\big).
    \end{split}
\end{equation}
The requirement that the curvature of the cable's profile curve is strictly limited by $\frac{1}{r_{i}}$ (i.e., $\norm{\alpha''(\eta)} \leq \frac{l_{i}^{2}}{r_{i}}$) ensures that the parametrization \eqref{eq:Gammalat} of the lateral surface is essentially injective. In practical terms, this means the cable is free of kinks.

Based on the construction we have previously outlined, the cable's shape is determined by the expression
\begin{equation}
  \Omega_{i} = \dset[\big]{\alpha_i(\eta) + \delta\beta_{i\eta}(\theta)}
  {(\delta,\eta,\theta)\in[0,1]^2 \times \lparen-\uppi,\uppi\rbrack}.
\end{equation}
Its boundary is given by 
\begin{equation*}
  \partial\Omega_{i} = \Gamma_{i,\textup{lat}}\cupdot\Gamma_{i,\textup{end}},
\end{equation*}
where $\Gamma_{i,\textup{end}}$ is the union of cross-sectional areas at the two ends of the $i$th cable and $\Gamma_{i,\textup{lat}}$ is the lateral surface of the $i$th cable. That is,
\begin{align*}
  \Gamma_{i,\textup{end}}
  &= \dset[\big]{\alpha_i(\eta) + \delta\beta_{i\eta}(\theta)}
     {(\delta,\eta,\theta) \in \sset{0,1} \times[0,1) \times (-\uppi,\uppi]}, \\
  \Gamma_{i,\textup{lat}}
  &= \dset[\big]{\alpha_{i}(\eta) + \beta_{i\eta}(\theta)}{(\eta,\theta) \in (0,1) \times \lparen -\uppi,\uppi\rbrack}.
\end{align*}
For 
$\xi\in\Gamma_{\textup{lat}}$, we below introduce the vectors $\tau_\textup{long}(\xi),\tau_\textup{perp}(\xi)\in\R^3$, referred to as unit vectors in longitudinal and perpendicular direction, resp., see \Cref{fig:cable}. We further consider the outward normal unit vector by $\nu(\xi)\in\R^3$. These three vectors are defined by
\begin{multline}
  \forall\;i=1,\ldots,k,\,\eta\in[0,1],\,\theta\in \lparen -\uppi,\uppi \rbrack:\\
  \tau_{\textup{long}}(\Phi_i(\eta,\theta))=\frac{\alpha_i^\prime(\eta)}{l_i},\;
  \tau_{\textup{perp}}(\Phi_i(\eta,\theta))=\frac{\beta_{i\eta}^\prime(\theta)}{r_i},\;
  \nu(\Phi_i(\eta,\theta))=\frac{\beta_{i\eta}(\theta)}{r_i},\label{eq:longvec}
\end{multline}
where $\beta_{i\eta}^\prime$ stands for the derivative with respect with respect to $\theta$.
Expectably, $\tau_\textup{long}(\xi)$ and $\tau_\textup{perp}(\xi)$ form an orthonormal basis of the tangent space of $\Gamma_{\textup{lat}}$ at $\xi$, and that $\nu(\xi)$ spans the normal space. This is proven in the sequel.

\begin{lemma}\label{th:orthonormal-basis-of-tangential-space}
Under the assumptions and notation made in this section, it holds that, for all $\xi\in \Gamma_{\textup{lat}}$,  
  $(\tau_{\textup{long}}(\xi),\nu(\xi),\tau_{\textup{perp}}(\xi))$ is an orthonormal basis of $\R^{3}$ with determinant $1$. Moreover, $(\tau_{\textup{long}}(\xi),\tau_{\textup{perp}}(\xi))$ spans the tangent space of $\Gamma_{\textup{lat}}$, and $\nu(\xi)$ is the outward normal unit vector of $\Omega_1\cup\cdots\cup\Omega_k$ at $\xi\in \Gamma_{\textup{lat}}$.
\end{lemma}

\begin{proof}
  Let $\xi\in \Gamma_{\textup{lat}}$. Then there exist $i\in\{1,\ldots,k\}$, $\eta\in[0,1]$, $\theta\in \lparen -\uppi,\uppi\rbrack$, such that $\xi=\Phi_i(\eta,\theta)$.
  It follows directly from their definition that all $\tau_{\textup{long}}(\xi)$, $\tau_{\textup{perp}}(\xi)$, $\nu(\xi))$ have length one. Further, by definition of $\beta_{i\eta}$, we have $\alpha_i^{\prime}(\eta) \perp \beta_{i\eta}(\theta)$, and thus also $\tau_{\textup{long}}(\xi)\perp \tau_{\textup{perp}}(\xi)$. Moreover, the property that $(\tau_{\textup{long}}(\xi),\nu(\xi),\tau_{\textup{perp}}(\xi))$ is an orthonormal basis with determinant $1$ follows from
  \begin{align*}
    \tau_{\textup{long}}(\xi)\times \nu(\xi)
    &= \tfrac{1}{l_{i}}\alpha_{i}^{\prime}(\eta) \times \tfrac{1}{r_{i}}\beta_{i\eta}(\theta)\\
     &= \cos(\theta) \underbrace{(\tfrac{1}{l_i}\alpha_i^{\prime}(\eta) \times \kappa_{i1}(\eta))}_{=\mathrlap{\kappa_{i2}(\eta)}}
       + \sin(\theta) \underbrace{(\tfrac{1}{l_i}\alpha_i^{\prime}(\eta) \times \kappa_{i2}(\eta))}_{=\mathrlap{-\kappa_{i1}(\eta)}}
    \\
    &= \cos(\theta) \kappa_{i2}(\eta) - \sin(\theta) \kappa_{i1}(\eta)\\&=\tfrac{1}{r_i}\beta_{i\eta}^{\prime}(\theta)=\tau_{\textup{perp}}(\xi).
  \end{align*}
  The tangent space of $\Gamma_{\textup{lat}}$ at $\xi$ is spanned by the two vectors
  \[\partial_{\eta} (\alpha_i(\eta) + \beta_{i\eta}(\theta)),\qquad\partial_{\theta} (\alpha_{i}(\eta) + \beta_{i\eta}(\theta)) = \beta_{i\eta}^{\prime}(\theta).\] Since, by \Cref{th:linear-dependencies-of-derivatives-of-coordinates}, $\partial_{\eta} \beta_{i\eta}(\theta)$ is in the span of $(\alpha_i^{\prime}(\eta),\beta_{i\eta}^{\prime}(\theta))$, we can conclude that the tangent space of $\Gamma_{\textup{lat}}$ at $\xi$ is spanned by $(\alpha_i^{\prime}(\eta),\beta_{i\eta}^{\prime}(\theta))$, and thus also by $(\tau_{\textup{long}}(\xi),\tau_{\textup{perp}}(\xi))$. As a~consequence, $\nu(\xi)=\tfrac{\beta_{i\eta}(\theta)}{r_i}$ is a normal vector of $\Gamma_{\textup{lat}}$. Moreover, $\Phi_i(\eta,\theta) - \beta_{i\eta}(\theta) = \alpha_i(\eta)$, which is clearly inside $\Omega_i$, which implies that $\beta_{i\eta}$ points outwards.
\end{proof}



\subsection{Initial and boundary conditions}\label{sec:max}

The evolution of the electromagnetic field is governed by Maxwell's equations~\eqref{eq:Maxwell}.  These are equipped with initial conditions
\[
  \bm{B}(0)=\bm{B}_0,\quad \bm{D}(0)=\bm{D}_0
\]
for some given $\bm{B}_0,\bm{D}_0\colon \Omega\to\C^3$. Throughout the remaining part, we omit the divergence conditions~\eqref{eq:divfree} on the electric and magnetic flux densities here.  
Given that $\div \rot = 0$, the divergence conditions are preserved over time and can therefore be enforced through the initial state.


To ensure a physically complete description, it is necessary to define appropriate boundary conditions.

The description of the domain $\Omega$ in \eqref{eq:domain-Omega} leads to the circumstance that $\partial\Omega$ is the disjoint union of the boundaries $\partial\Omega_{0}, \partial\Omega_1,\ldots, \partial\Omega_k$, Whereas the boundaries of the cable shapes themselves are the disjoint union of the lateral surfaces $\Gamma_{i,\textup{lat}}$ and the cover surfaces $\Gamma_{i,\textup{end}}$, $i=1,\ldots,k$. That is,
\begin{align}\label{eq:bndsplit}
\partial\Omega=\partial\Omega_0\cupdot\Gamma_{\textup{end}}\cupdot\Gamma_{\textup{lat}},
  \quad\text{where}\quad
  \Gamma_{\textup{end}}=\bigcupdot_{i=1}^{k}\Gamma_{i,\textup{end}},\; \Gamma_{\textup{lat}}=\bigcupdot_{i=1}^{k}\Gamma_{i,\textup{lat}}.
\end{align}
Subsequently, we introduce the boundary conditions on $\partial\Omega_0$ and $\Gamma_{\textup{end}}$. The conditions on $\Gamma_{\textup{lat}}$ are determined by the coupling relations between the electromagnetic field and the transmission lines, which will only be presented in the section after this one.

Any of these boundary conditions uses the outward normal vector $\nu(\xi)\in\R^3$, which is well-defined for almost every $\xi\in\partial\Omega$ (with respect to the surface measure on $\partial\Omega$), due to $\Omega$ being a Lipschitz domain. 
This property also implies that $\nu\colon \partial\Omega\to\C^3$ is measurable and essentially bounded.
We also introduce the tangential orthogonal projection at $\xi\in\partial\Omega$, which is defined by
\begin{equation}
\begin{aligned}
  \tantr(\xi)\colon
  \mapping{\C^3}{\C^3}{w}{(\nu(\xi)\times w)\times \nu(\xi).}
\end{aligned}\label{eq:tangproj}
\end{equation}
Once again, $\tantr(\xi)$ is well-defined for almost all $\xi\in\partial\Omega$. Using the expression
\begin{multline*}
  \tantr(\xi)w = (\nu(\xi)\times w)\times \nu(\xi) = w-(\nu(\xi)^\top w)\nu(\xi) \\
  \forall\,w\in\C^3 \text{ and almost all }\xi\in\partial\Omega,
\end{multline*}
we have, for almost all $\xi\in\partial\Omega$,  that $\tantr(\xi)$ is an orthogonal projector onto the tangent space of $\partial\Omega$ at $\xi$. It is evident that for almost all $\xi\in\partial\Omega$,
\[\forall\,w\in\C^3:\quad \tantr(\xi) w=0 \,\Leftrightarrow\,\nu(\xi)\times w=0.\]

\paragraph{Boundary conditions at the computational domain}
We impose the boundary condition 
\begin{equation}
  \tantr \bm{E}(t)\big\vert_{\partial\Omega_0} = 0,\; t\geq 0.\label{eq:el0}
\end{equation}
This condition represents perfect electrical insulation outside of $\Omega_0$. Alternatively, we can model superconductivity outside of $\Omega$ by applying the boundary condition
\begin{equation}
  \nu \times \bm{H}(t)\big\vert_{\partial\Omega_0} = 0,\; t\geq 0.\label{eq:mag0}
\end{equation}
Note that the relations $\bm{E}(t) = \bm{\epsilon}^{-1} \bm{D}(t)$ and $\bm{H}(t) = \bm{\mu}^{-1} \bm{B}(t)$ imply corresponding boundary conditions for $\bm{D}$ and $\bm{B}$.
Both of the above boundary conditions result in a physically closed system, leading to the reflection of electromagnetic waves at the boundary. To mitigate this effect, we can consider the ``Leontovich boundary condition'' \cite{WeSt13} (also found in the original reference \cite{Leo44}), which is defined as
\[
  \tantr \bm{E}(t)\big\vert_{\partial\Omega_0} + r\big(\nu\times(\bm{H}(t))\big\vert_{\partial\Omega_0}\big) = 0, \quad t\geq 0
\]
where $r\colon \partial\Omega\to\C^{3\times 3}$ is measurable, essentially bounded, and $r(\xi)+r(\xi)\adjun$ is positive semi-definite for almost all $s\in\partial\Omega$. As they give rise to absorption of electromagnetic waves, the Leontovich boundary conditions do not maintain the physical closedness of the domain $\Omega_0$. They result in power loss at the boundary.

Additionally, we should mention that the boundary conditions can be combined. Specifically, $\partial\Omega$ can be divided into three Lipschitz submanifolds, each subject to one of the above types of boundary conditions. 
For sake of brevity, we present our analysis only for the boundary conditions representing perfect electrical insulation outside $\Omega_0$. However, our results can be suitably adapted to accommodate the aforementioned other types of boundary conditions, as well as their combinations. We emphasize that any form of boundary condition at $\partial\Omega_0$ is obsolete, if we choose the entire $\R^3$ as the computational domain, for obvious reasons.


\paragraph{Boundary conditions at the cover surfaces of the cable}
We assume that the electric field is polarized in normal direction to the cover surfaces of the cable. This leads to the boundary condition
\begin{equation}
  \tantr\bm{E}(t)\big\vert_{\Gamma_{\textup{end}}}=0,\;t\geq0.\label{eq:bndlat0}
\end{equation}

\subsection{Power balance}\label{sec:enbalem}

For a~general Lipschitz domain $\Omega\subset\R^3$ the rot operator satisfies the following integration by parts formula for all sufficiently smooth $f,g\colon \Omega\to\C^3$:
\[\int_\Omega \big(\rot f(\xi)\big)\adjun  g(\xi)- f(\xi)\adjun  \big(\rot g(\xi)\big)\dx[\xi]=\int_{\partial\Omega}  \big(\nu(\xi)\times  f(\xi)\big)\adjun \big( \tantr(\xi) f(\xi)\big)\dx[\xi],\]
where the latter is an integral with respect to the two-dimensional surface measure.
For a rigorous derivation and the use of appropriate function spaces, we refer the reader to \cite{BuCoSh02, PhReSc23}.

Next we derive the power balance for Maxwell's equations. 
Given are the magnetic flux density $\bm{B}(t)$ and electric flux density $\bm{D}(t)$ at time $t$. The magnetic energy density at $\xi \in \Omega$ is $\tfrac{1}{2} \bm{B}(t,\xi)\adjun \bm{\mu}(\xi)^{-1} \bm{B}(t,\xi)$, the electric energy density is $\tfrac{1}{2} \bm{D}(t,\xi)\adjun \bm{\epsilon}(\xi)^{-1} \bm{D}(t,\xi)$. Consequently, the total energy of the electromagnetic field at time $t$ is given by the sum of the spatial integrals of these energy densities over $\Omega$, i.e.,
\[\mathcal{E}_{\rm em}(\bm{B}(t),\bm{D}(t))=\frac12\int_{\Omega}\bm{B}(t,\xi)\adjun \bm{\mu}(\xi)^{-1} \bm{B}(t,\xi)+\bm{D}(t,\xi)\adjun \bm{\epsilon}(\xi)^{-1} \bm{D}(t,\xi) \dx[\xi].\]
To analyze the power in the system, we differentiate the total energy with respect to time. By using  $\bm{H}(t) \coloneqq \bm{\mu}^{-1}\bm{B}(t)$ and $\bm{E}(t) \coloneqq \bm{\epsilon}^{-1}\bm{D}(t)$, we have
\begin{align*}
    \MoveEqLeft
    \ddts
    \mathcal{E}_{\rm em}(\bm{B}(t),\bm{D}(t))\\
    &=\Re\int_{\Omega}\big(\ddts\bm{B}(t,\xi)\big)\adjun \bm{H}(t,\xi)+\bm{E}(t,\xi)\adjun  \ddts\bm{D}(t,\xi) \dx[\xi]\\
    &=\Re\int_{\Omega}-\big(\rot\bm{E}(t,\xi)\big)\adjun  \bm{H}(t,\xi)+\bm{E}(t,\xi)\adjun \big(\rot\bm{H}(t,\xi)\big) \dx[\xi]\\
    &\quad -\Re\int_{\Omega}\bm{E}(t,\xi)\adjun \sigma(\xi) \bm{E}(t,\xi)\dx[\xi]\\
&=\Re\int_{\partial\Omega}\big(\tantr(\xi)\bm{E}(t,\xi)\big)\adjun  \big(\nu(\xi)\times\bm{H}(t,\xi)\big) \dx[\xi]\\
    &\quad -\Re\int_{\Omega}\bm{E}(t,\xi)\adjun \sigma(\xi) \bm{E}(t,\xi)\dx[\xi]
\end{align*}
For the geometry described in \Cref{sec:Omega}, the boundary of $\Omega$ is partitioned as in \eqref{eq:bndsplit}, and thus
\begin{align*}
\MoveEqLeft
\int_{\partial\Omega}\big(\tantr(\xi)\bm{E}(t,\xi)\big)\adjun  \big(\nu(\xi)\times\bm{H}(t,\xi)\big) \dx[\xi]\\
&=\int_{\partial\Omega_0}\big(\tantr(\xi)\bm{E}(t,\xi)\big)\adjun  \big(\nu(\xi)\times\bm{H}(t,\xi)\big) \dx[\xi]\\
&\quad+\int_{\Gamma_{\textup{end}}}\big(\tantr(\xi)\bm{E}(t,\xi)\big)\adjun  \big(\nu(\xi)\times\bm{H}(t,\xi)\big) \dx[\xi]\\
&\quad+\int_{\Gamma_{\textup{lat}}}\big(\tantr(\xi)\bm{E}(t,\xi)\big)\adjun  \big(\nu(\xi)\times\bm{H}(t,\xi)\big) \dx[\xi].
\end{align*}
In any of the cases \eqref{eq:el0} or \eqref{eq:mag0}, we have that the first summand vanishes. Further, by \eqref{eq:bndlat0}, the electric field intensity is normal to the cover surfaces of the cables. Therefore, the power balance simplifies to  
\begin{align*}
    &\phantom{=}\ddts
    \mathcal{E}_{\rm em}(\bm{B}(t),\bm{D}(t))\\
&=\Re\int_{\Gamma_{\textup{lat}}}\big(\tantr(\xi)\bm{E}(t,\xi)\big)\adjun  \big(\nu(\xi)\times\bm{H}(t,\xi)\big) \dx[\xi] -\Re\int_{\Omega}\bm{E}(t,\xi)\adjun \sigma(\xi) \bm{E}(t,\xi)\dx[\xi].
\end{align*}
The term
\begin{equation}\Re\int_{\Omega}\bm{E}(t,\xi)\adjun \sigma(\xi) \bm{E}(t,\xi)\dx[\xi]\label{eq:emloss}\end{equation}
which is nonnegative due to the assumption that $\sigma + \sigma^*$ is pointwise positive semi-definite, represents the power dissipated due to spatially distributed damping. Further, the term
\[\Re\int_{\Gamma_{\textup{lat}}}\big(\tantr(\xi)\bm{E}(t,\xi)\big)\adjun  \big(\nu(\xi)\times\bm{H}(t,\xi)\big) \dx[\xi]\]
is the power supplied to the electromagnetic field at the lateral surface of the cables.

\section{Coupling - transmission line and lateral cable surfaces}\label{sec:coupling-idea}

To complete our model of the radiating curved cables, we now introduce coupling relations between the electric and magnetic field intensities at the lateral surfaces of the cables. These surfaces serve as the interface between the external electromagnetic field (modeled by Maxwell's equations~\eqref{eq:Maxwell}) and the transmission line (modeled by the telegrapher's equations~\eqref{eq:teleq}). Mathematically, this requires mapping certain functions defined on a two-dimensional spatial domain (the lateral surface of the cable) to functions defined on the one-dimensional spatial domain $[0,1]$, and vice versa.

We begin with a physical derivation of the coupling conditions. Specifically, we independently derive, on the one hand, coupling conditions between the externally applied electric field $\bm{E}_{\ext}$ at the transmission lines and the tangential component of the electric field intensity of the electromagnetic field at the lateral surfaces of the cables, and, on the other hand, between the external current $\bm{I}_{\ext}$ at the transmission lines and the tangential magnetic field intensity at the same surfaces.

Subsequently, we establish a mathematical relation between these two types of coupling. In particular, we show that each can be represented by an operator, and that these two operators are adjoint to each other. This property will later be used to show that the overall coupled system satisfies a~power balance.

\subsection{Physical derivation}

We start with the coupling of the external current with the magnetic field intensity. Thereafter, we will consider the electric coupling.




\paragraph{External current and magnetic field}
The coupling between electric current and magnetic field intensity is elegantly expressed by Amp{\`e}re's law, which states that the line integral of the magnetic field around a closed loop surrounding a current-carrying conductor is equal to the total current enclosed. In mathematical terms, for $i=1,\ldots,k$, this is described by the line integral
\[\forall \, \eta\in[0,1] : \qquad 
\oint_{\alpha_{i}(\eta) + \beta_{i\eta}} \mspace{-8mu} \bm{H}(t,\xi) \cdot \dx[\bm{s}(\xi)]=-\bm{I}_{i,\ext}(t,\eta),\]
where $\bm{I}_{i,\ext}$ denotes the $i$\textsuperscript{th} component of $\bm{I}_{\ext}(t,\eta)$ (the input that corresponds to the $i$\textsuperscript{th} cable). Note that the negative sign arises from the fact that the path $\beta_{i\eta}$ is oriented clockwise.

Since tangents of the path $\alpha_{i}(\eta) + \beta_{i\eta}$ are perpendicular to the cable, it follows that this is the same as the line integral of the tangential projection of the magnetic field intensity of $\bm{H}$ over $\alpha_{i}(\eta) + \beta_{i\eta}$. Now using \eqref{eq:tangproj}, we see that Amp{\`e}re's law is equivalent to 
\begin{equation}
\forall \, \eta\in[0,1] : \qquad 
\oint_{\alpha_{i}(\eta) + \beta_{i\eta}} \mspace{-8mu} \big(\nu(\xi)\times \bm{H}(t,\xi)\big)\times \nu(\xi) \cdot \dx[\bm{s}(\xi)]=-\bm{I}_{i,\ext}(t,\eta).
\label{eq:magcouple}
\end{equation}

\paragraph{Electric field}
Next we describe the coupling between the quantity $\bm{E}_{\ext}$ in the transmission line with boundary values of the exterior electric field. Here we make use of the fact that the voltage across a~curve along a~cross sectional area of the cable is constant. Further, the voltage across a~curve longitudinal to the cable from $\eta_0\in [0,1]$ to $\eta_1\in [0,1]$ is given by $\int_{\eta_0}^{\eta_1}\bm{E}_{\ext}(t,\eta)\dx[\eta]$. Now using that the voltage across a~curve is given by the curve integral of the electrical field intensity, we are led to, for $i=1,\ldots,k$,
\[
  \forall \, \eta\in[0,1], \; \theta_0,\theta_1\in \lparen -\uppi,\uppi \rbrack :
  \quad
  \int_{\theta_0}^{\theta_1} \big(\tfrac{\partial}{\partial \theta}\Phi_i(\eta,\theta)\big)\trans \bm{E}(t,\Phi_i(\eta,\theta)) \dx[\theta]=0,
\]
and
\begin{multline*}
  \forall \, \theta\in \lparen -\uppi,\uppi \rbrack,\;\eta_0,\eta_1\in [0,1]:\\ \int_{\eta_0}^{\eta_1}\big(\tfrac{\partial}{\partial \eta} \Phi_i(\eta,\theta)\big)\trans \bm{E}(t,\Phi_i(\eta,\theta)) \dx[\eta]
  = \int_{\eta_0}^{\eta_1}\bm{E}_{i,\ext}(t,\eta) \dx[\eta],
\end{multline*}
where $\bm{E}_{i,\ext}$ denotes the $i$\textsuperscript{th} component of $\bm{E}_{\ext}(t,\eta)$ (the output that corresponds to the $i$\textsuperscript{th}).
Hence, by defining the gradient as the transpose of the Jacobian, we obtain
\[
  \forall \, \theta\in \lparen -\uppi,\uppi \rbrack,\;\eta\in [0,1]:
  \nabla \Phi_i(\eta,\theta) \bm{E}(t,\Phi_i(\eta,\theta))
  =\begin{pmatrix}
    \bm{E}_{i,\ext}(t,\eta) \\ 0
  \end{pmatrix}.
\]
Now multiplying from the left with the Moore-Penrose inverse of $\nabla \Phi_i(\eta,\theta)$, we obtain that
\begin{multline}
  \forall \, \theta\in \lparen -\uppi,\uppi \rbrack,\; \eta\in [0,1]: \\
  \tantr(\Phi_i(\eta,\theta)) \bm{E}(t,\Phi_i(\eta,\theta))
  = \nabla \Phi_i(\eta,\theta)^{\dagger}
  \begin{pmatrix}
    \bm{E}_{i,\ext}(t,\eta)\\0
  \end{pmatrix},\label{eq:elcouple}
\end{multline}
where $\tantr(\xi)\in \R^{3\times 3}$ is the orthogonal projection onto the tangent space of $\partial\Omega$ at $\xi\in \partial\Omega$.

\paragraph{Overall system}
The above relations---on the one hand, between the tangential trace of the magnetic field and $\bm{I}_{\ext}$, and on the other hand, between the tangential trace of the electric field and $\bm{E}_{\ext}$---can be described by suitable coupling operators $\portOp_\textup{mag}$ and $\portOp_\textup{el}$, such that  
\begin{equation}
-\bm{I}_{\ext}(t) = \portOp_\textup{mag} \big(\nu \times \bm{H}(t)\big), \quad \tantr \bm{E}(t) = \portOp_\textup{el} \bm{E}_{\ext}(t).\label{eq:couplcondop}
\end{equation}
A more precise functional-analytic description will be provided in the following subsection.
Regarding $\bm{I}_{\ext}$ as an input of the transmission line and $\bm{E}_{\ext}$ as an output, our entire configuration can be schematically represented by \Cref{fig:coupled-diagram}.

\begin{figure}[htbp]
  \centering
  \begin{subfigure}[t]{0.47\linewidth}
    \centering
    \begin{tikzpicture}[scale=0.8]
      \tikzstyle{mynode1} = [rectangle, minimum width=2.5cm, minimum height=0.8cm, text centered, draw=black]
      \tikzstyle{mynode2} = [rectangle, minimum width=0.8cm, minimum height=0.8cm, text centered, draw=black]

      \coordinate (vh) at (0,-2); 
      \coordinate (vv) at (2,0); 

      \node[mynode1] (H3) at (0,0) {\parbox[c]{2cm}{\centering Transmission\\line}};
      \node[mynode1] (H1) at ($(H3) + 2*(vh)$) {\parbox[c]{2cm}{\centering Maxwell's\\equations}};


      \coordinate (A) at ($(H3) + (vh) + (vv)$);
      \coordinate (B) at ($(H3) + (vh) - (vv)$);

      \draw[-Latex] ($(H3|- H3.175) - (vv) + (-1,0)$) -- node[above] {$u$} (B|- H3.175) -- (H3.175);

      \draw[Latex-] ($(H3|- H3.5) + (vv) + (1,0)$) -- node[above] {$y$} (A|- H3.5) -- (H3.5);

      \draw[-Latex] ($(H3|- H3.185) - (vv) + (-1,0)$) -- node[below] {$\bm{I}_{\ext}$} (B|- H3.185) -- (H3.185);

      \draw[Latex-] ($(H3|- H3.-5) + (vv) + (1,0)$) -- node[below] {$\bm{E}_{\ext}$} (A|- H3.-5) -- (H3.-5);

      \draw[Latex-] ($(H1) - (vv) + (-1,0)$) -- node[above] {$\nu\times\bm{H}\quad\quad$} (H1);
      \draw[-Latex] ($(H1) + (vv) + (1,0)$) -- node[above] {$\tantr \bm{E}$}  (H1);
    \end{tikzpicture}
    \caption{\label{fig:uncoupled-diagram}Uncoupled ports}
  \end{subfigure}%
  \hfill
  \begin{subfigure}[t]{0.47\linewidth}
    \centering
    \begin{tikzpicture}[scale=0.8]
      \tikzstyle{mynode1} = [rectangle, minimum width=2.5cm, minimum height=0.8cm, text centered, draw=black]
      \tikzstyle{mynode2} = [rectangle, minimum width=0.8cm, minimum height=0.8cm, text centered, draw=black]

      \coordinate (vh) at (0,-2); 
      \coordinate (vv) at (2,0); 

      \node[mynode1] (H3) at (0,0) {\parbox[c]{2cm}{\centering Transmission\\line}};
      \node[mynode1] (H1) at ($(H3) + 2*(vh)$)  {\parbox[c]{2cm}{\centering Maxwell's\\equations}};

      \node[mynode2] (A) at ($(H3) + (vh) + (vv)$) {$\;\;\portOp_\textup{el}\;\;$};
      \node[mynode2] (B) at ($(H3) + (vh) - (vv)$) {$-\portOp_\textup{mag}$};

      \draw[-Latex] (H1.180) -- (B |- H1.180) -- node[left] {$\nu\times\bm{H}$} (B);
      \draw[-Latex] (B) -- node[left] {$\bm{I}_{\ext}$}(B |- H3.185) -- (H3.185); 

      \draw[-Latex] (H3.-5) -- (A |- H3.-5) -- node[right] {$\bm{E}_{\ext}$} (A);
      \draw[-Latex] (A) -- node[right] {$\tantr \bm{E}$} (A |- H1.0) -- (H1.0); 

      \draw[-Latex] ($(H3|- H3.175) - (vv) + (-1,0)$) -- node[above] {$u$} (B|- H3.175) -- (H3.175);
      \draw[Latex-] ($(H3|- H3.5) + (vv) + (1,0)$) -- node[above] {$y$} (A|- H3.5) -- (H3.5);

    \end{tikzpicture}
    \caption{\label{fig:coupled-diagram}Coupled ports}
  \end{subfigure}
  \caption{\label{fig:coupling-process}(Un)coupled ports}
\end{figure}
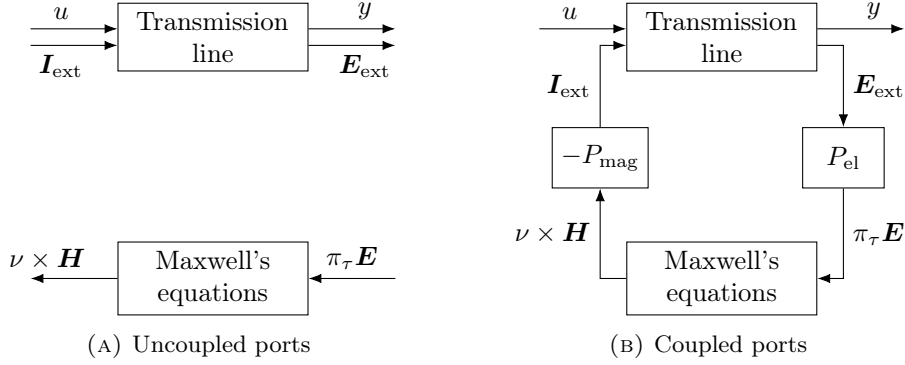

%
%
%
%
%
%
%
%

\subsection{Analysis of the coupling conditions}\label{sec:analysis-of-port-operator}

In order to show that the coupled field–cable system satisfies a~power balance, we now present some analytical details of the coupling relations introduced just before. Specifically, we analyze the coupling operators $\portOp_\textup{mag}$ and $\portOp_\textup{el}$ and discuss their relationship to each other.


Namely, for the parameterization $\Phi_i$ of the
lateral boundary $\Gamma_{i,\textup{lat}}$ of the $i$\textsuperscript{th} (see \eqref{eq:Gammalat}), we denote the inverse by $\Psi_i$. That is,
\begin{align*}
\Psi_{i}&\colon
\mapping{\Gamma_{i,\textup{lat}}}{[0,1]\times \lparen -\uppi,\uppi \rbrack}{\Phi(\eta,\theta)}{(\eta,\theta),}
\\
\Psi_{i1}&\colon 
\mapping{\Gamma_{i,\textup{lat}}}{[0,1]}{\Phi(\eta,\theta)}{\eta.}
\end{align*}
We
introduce the operators
\begin{align}
  \portOp_{i,\textup{mag}} \colon&\;
  \mapping{\Lp{2}(\Gamma_{i,\textup{lat}};\C^{3})}{\Lp{2}((0,1);\C)}{g}{%
    \bigg(\eta \mapsto\displaystyle \vphantom{\oint}\smash{\oint\limits_{\mathclap{\alpha_{i}(\eta) + \beta_{i\eta}}}} \mspace{8mu} g(\xi) \times \nu(\xi) \cdot \dx[\bm{s}(\xi)]\bigg),}
    \vphantom{\mapping{\Lp{2}\C^{3}}{\Lp{2}}{}{\displaystyle\oint\limits_{\beta_{i\eta}}}}\label{eq:magiop}\\[3mm]
  \portOp_{i,\textup{el}} \colon&\;
  \mapping{\Lp{2}((0,1);\C)}{\Lp{2}(\Gamma_{i,\textup{lat}};\C^{3})}{f}{(\grad \Phi_{i}\circ\Psi_i)^{\dagger} 
    \begin{pmatrix}
      f \circ \Psi_{i1}\\
      0
    \end{pmatrix}.
  }\label{eq:eliop}
\end{align}
\;
Well-definition of $\portOp_{\textup{mag},i}$ follows by \Cref{le:port-operator-well-defined} (basically by Fubini's theorem).
The overall port operators are then given by 
\begin{align}
  \portOp_{\textup{mag}} \colon&\;
  \mapping{\Lp{2}(\Gamma_{\textup{lat}};\C^{3})}{\Lp{2}((0,1);\C^k)}{g}{%
    \begin{pmatrix}
        \portOp_{1,\textup{mag}}\Big(\left.g\right|_{\Gamma_{1,\textup{lat}}}\Big)\\\vdots\\\portOp_{k,\textup{mag}}\Big(\left.g\right|_{\Gamma_{k,\textup{lat}}}\Big)
    \end{pmatrix},}\label{eq:magop}\\[4mm]
  \portOp_{\textup{el}} \colon&\;
  \mapping{\Lp{2}((0,1);\C^k)}{\Lp{2}(\Gamma_{\textup{lat}};\C^{3})}{\begin{pmatrix}f_1\\\vdots\\f_k\end{pmatrix}}{%
 g\text{ with }\left.g\right|_{\Gamma_{i,\textup{lat}}}= \portOp_{i,\textup{el}}f_i \;\forall\, i=1,\ldots,k,}\label{eq:elop}
\end{align}
and the coupling conditions \eqref{eq:magcouple} and \eqref{eq:elcouple} can be reformulated to \eqref{eq:couplcondop}.

In the sequel, we show that $\portOp_{\textup{mag}}$ and $\portOp_{\textup{el}}$ are adjoint to each other,
which means that the coupling scheme depicted in \Cref{fig:coupled-diagram} reduces to that shown in \Cref{fig:coupled-diagram_adj}.

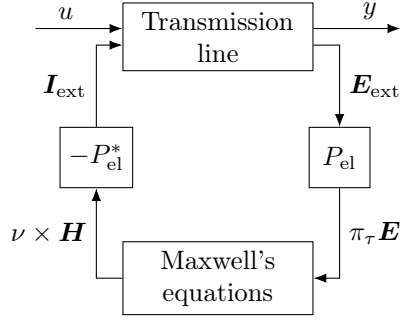
\begin{figure}[htbp]
  \centering
    \begin{tikzpicture}[scale=0.8]
      \tikzstyle{mynode1} = [rectangle, minimum width=2.5cm, minimum height=0.8cm, text centered, draw=black]
      \tikzstyle{mynode2} = [rectangle, minimum width=0.8cm, minimum height=0.8cm, text centered, draw=black]

      \coordinate (vh) at (0,-2); 
      \coordinate (vv) at (2,0); 

      \node[mynode1] (H3) at (0,0) {\parbox[c]{2cm}{\centering Transmission\\line}};
      \node[mynode1] (H1) at ($(H3) + 2*(vh)$)  {\parbox[c]{2cm}{\centering Maxwell's\\equations}};

      \node[mynode2] (A) at ($(H3) + (vh) + (vv)$) {$\portOp_\textup{el}$};
      \node[mynode2] (B) at ($(H3) + (vh) - (vv)$) {$-\portOp_\textup{el}\adjun$};

      \draw[-Latex] (H1.180) -- (B |- H1.180) -- node[left] {$\nu\times\bm{H}$} (B);
      \draw[-Latex] (B) -- node[left] {$\bm{I}_{\ext}$}(B |- H3.185) -- (H3.185);  \draw[-Latex] (H3.-5) -- (A |- H3.-5) -- node[right] {$\bm{E}_{\ext}$} (A);
      \draw[-Latex] (A) -- node[right] {$\tantr \bm{E}$} (A |- H1.0) -- (H1.0);
      \draw[-Latex] ($(H3|- H3.175) - (vv) + (-1,0)$) -- node[above] {$u$} (B|- H3.175) -- (H3.175);
      \draw[Latex-] ($(H3|- H3.5) + (vv) + (1,0)$) -- node[above] {$y$} (A|- H3.5) -- (H3.5);
    \end{tikzpicture}
    \caption{\label{fig:coupled-diagram_adj}Coupled ports}
\end{figure}


Our construction shows that
$
\portOp_{i,\textup{el}} f =(\grad \Phi_{i}\circ\Psi_i)^{\dagger} 
    \begin{psmallmatrix}
      f \circ \Psi_{i1}\\
      0
    \end{psmallmatrix}
$
is pointwise a scalar multiple of the tangential vector $\tau_{\textup{long}}$ in the longitudinal direction of the cable; see \Cref{fig:cable}. The following lemma makes this precise by providing the stretching factor.

\begin{lemma}\label{th:explicit-form-of-P}
For $i=1,\ldots,k$, let $\tau_{i,\textup{long}}\colon \Gamma_{\textup{lat}}\to\R^3$ be the unit vector in longitudinal direction of the cable, see \eqref{eq:longvec}.
The operator $P_{i,\textup{el}}$ as in \eqref{eq:eliop} fulfills 
  \begin{multline*}
\forall\;i=1,\ldots,k,\;f\in \Lp{2}((0,1);\C):\\
    P_{i,\textup{el}} f = (\grad \Phi_{i})^{\dagger} 
    \begin{pmatrix}
      f \circ \Psi_{i} \\
      0
    \end{pmatrix}
    =
    \frac{1}{r_{i}} \frac{1}{\sqrt{\det \dd \Phi_{i}\trans \dd \Phi_{i}}} f \tau_{i,\textup{long}}.
  \end{multline*}
\end{lemma}

\begin{proof}
  In order to simplify the notation we will drop everywhere the index $i$ as this is only dead weight in this context. Moreover, we drop the index in $\scprod{\argdot}{\argdot}_{\R^{3}}$ and just write $\scprod{\argdot}{\argdot}$ for the inner product in $\R^{3}$ in the following.
  \begin{steps}
    \item We will first assume that we have a special $\Phi$ denoted by $\tilde{\Phi} \colon [0,l] \times \lparen r\uppi,-r\uppi \rbrack \to \Gamma_{\textup{lat}}$ such that $\tilde{\Phi}(\eta,\theta) = \tilde{\alpha}(\eta) + \tilde{\beta}_{\eta}(\theta)$, where
    \begin{equation*}
      \norm{\tilde{\alpha}^{\prime}(\eta)} = 1
      \quad\text{and}\quad
      \norm{\tilde{\beta}_{\eta}^{\prime}(\theta)} = 1
      \quad\text{for all}\quad \eta, \theta
    \end{equation*}
    i.e., the difference is that $\tilde{\alpha}$ and $\tilde{\beta_{\eta}}$ are parameterized by their path length.
    Since $\tilde{\alpha}^{\prime} \perp \tilde{\beta}_{\eta}$ we also have $\tilde{\alpha}^{\prime} \perp \tilde{\beta}_{\eta}^{\prime}$ (in the following we use $\scprod{\argdot}{\argdot}$ as the inner product in $\R^{3}$), which is used for (the second equality)
    \begin{align*}
      \dd \tilde{\Phi}
      &=
      \begin{pmatrix} \tilde{\alpha}^{\prime} + \partial_{\eta} \tilde{\beta}_{\eta} & \tilde{\beta}_{\eta}^{\prime}\end{pmatrix}
      \\
      \dd \tilde{\Phi}\trans \dd \tilde{\Phi}
      &=
      \begin{pmatrix}
        1 + 2\scprod{\tilde{\alpha}^{\prime}}{\partial_{\eta} \tilde{\beta}_{\eta}} + \scprod{\partial_{\eta} \tilde{\beta}_{\eta}}{\partial_{\eta} \tilde{\beta}_{\eta}}
        & \scprod{\partial_{\eta} \tilde{\beta}_{\eta}}{\tilde{\beta}_{\eta}^{\prime}} \\
        \scprod{\partial_{\eta} \tilde{\beta}_{\eta}}{\tilde{\beta}_{\eta}^{\prime}} &
        1
      \end{pmatrix}.
    \end{align*}
    In order to calculate the determinant of $\dd \tilde{\Phi}\trans \dd \tilde{\Phi}$ we note that the columns of $\dd \tilde{\Phi}$ span the tangent space of $\Gamma_{\textup{lat}}$.
    By \Cref{th:orthonormal-basis-of-tangential-space} also $\tilde{\alpha}^{\prime}$ and $\tilde{\beta}_{\eta}^{\prime}$ span the tangent space of $\Gamma_{\textup{lat}}$.
    Moreover, $\tilde{\alpha}^{\prime}$ and $\tilde{\beta}_{\eta}^{\prime}$ are orthogonal, hence they form a orthonormal basis of the tangent space.
    This yields
    \begin{align*}
      \partial_{\eta} \tilde{\beta}_{\eta} &= \scprod{\partial_{\eta} \tilde{\beta}_{\eta}}{\tilde{\alpha}^{\prime}} \tilde{\alpha}^{\prime} + \scprod{\partial_{\eta} \tilde{\beta}_{\eta}}{\tilde{\beta}_{\eta}^{\prime}} \tilde{\beta}_{\eta}^{\prime}
      \\
      \scprod{\partial_{\eta} \tilde{\beta}_{\eta}}{\partial_{\eta} \tilde{\beta}_{\eta}} &= \scprod{\partial_{\eta} \tilde{\beta}_{\eta}}{\tilde{\alpha}^{\prime}}^{2} + \scprod{\partial_{\eta} \tilde{\beta}_{\eta}}{\tilde{\beta}_{\eta}^{\prime}}^{2}.
    \end{align*}
    Hence, we have
    \begin{align*}
      \det (\dd \tilde{\Phi}\trans \dd \tilde{\Phi})
      &= 1 + 2 \scprod{\partial_{\eta}\tilde{\beta}_{\eta}}{\tilde{\alpha}^{\prime}} + \scprod{\partial_{\eta}\tilde{\beta}_{\eta}}{\partial_{\eta}\tilde{\beta}_{\eta}} - \scprod{\partial_{\eta} \tilde{\beta}_{\eta}}{\tilde{\beta}_{\eta}^{\prime}}^{2}
      \\
      &= 1 + 2 \scprod{\partial_{\eta}\tilde{\beta}_{\eta}}{\tilde{\alpha}^{\prime}} + \scprod{\partial_{\eta}\tilde{\beta}_{\eta}}{\tilde{\alpha}^{\prime}}^{2}
      \\
      &=
      (1 + \scprod{\partial_{\eta} \tilde{\beta}_{\eta}}{\tilde{\alpha}^{\prime}})^{2}.
    \end{align*}
    Since $(\grad \tilde{\Phi})^{\dagger} = \dd \tilde{\Phi} (\dd \tilde{\Phi}\trans \dd \tilde{\Phi})^{-1}$ and we can easily compute the inverse of a $2 \times 2$ matrix, we have
    \begin{align*}
      (\grad \tilde{\Phi})^{\dagger}
      \begin{pmatrix} f \\ 0 \end{pmatrix}
      &=
      \begin{pmatrix} \tilde{\alpha}^{\prime} + \partial_{\eta} \tilde{\beta}_{\eta} & \tilde{\beta}_{\eta}^{\prime} \end{pmatrix}
      \frac{1}{\det(\dd \tilde{\Phi}\trans \dd \tilde{\Phi})}
      \begin{pmatrix} f \\ -\scprod{\partial_{\eta} \tilde{\beta}_{\eta}}{\tilde{\beta}_{\eta}^{\prime}} f\end{pmatrix}
      \\
      &=
      \frac{f}{(1 + \scprod{\partial_{\eta}\tilde{\beta}_{\eta}}{\tilde{\alpha}^{\prime}})^{2}}\big(\tilde{\alpha}^{\prime} + \underbrace{\partial_{\eta} \tilde{\beta}_{\eta} - \scprod{\partial_{\eta} \tilde{\beta}_{\eta}}{\tilde{\beta}_{\eta}^{\prime}} \tilde{\beta}_{\eta}^{\prime}}_{= \mathrlap{\scprod{\partial_{\eta}\tilde{\beta}_{\eta}}{\tilde{\alpha}^{\prime}} \tilde{\alpha}^{\prime}}}\big)
      \\
      &=
      \frac{f}{(1 + \scprod{\partial_{\eta}\tilde{\beta}_{\eta}}{\tilde{\alpha}^{\prime}})} \tilde{\alpha}^{\prime}
      = \frac{1}{\sqrt{\det(\dd \tilde{\Phi}\trans \dd \tilde{\Phi})}} f\,\tau_{\textup{long}}
    \end{align*}

    \item For the ``general'' $\Phi$, i.e., $\Phi(\eta,\theta) = \alpha(\eta) + \beta_{\eta}(\theta)$ such that $\norm{\alpha^{\prime}(\eta)} = l$ and $\norm{\beta_{\eta}^{\prime}(\theta)} = r$, we can define $\tilde{\Phi}$ by applying a stretching transform $\tilde{\Phi}(\eta,\beta) \coloneqq \Phi(\tfrac{1}{l}\eta,\frac{1}{r}\theta)$. Then $\tilde{\Phi}$ is such as in the first step.
    This gives
    \begin{equation*}
      \dd \Phi = \dd \tilde{\Phi}
      \begin{pmatrix} l & 0 \\ 0 & r \end{pmatrix}
      \quad\text{and}\quad
      \det(\dd \tilde{\Phi}\trans \dd \tilde{\Phi}) = \frac{1}{r^{2}l^{2}}\det(\dd {\Phi}\trans \dd {\Phi})
    \end{equation*}
    Applying the previous identities and the first step gives
    \begin{align*}
      (\grad \Phi)^{\dagger} \begin{pmatrix} f \\ 0\end{pmatrix}
      &=
      \dd \tilde{\Phi}
      \begin{pmatrix} l & 0 \\ 0 & r\end{pmatrix}
      \begin{pmatrix} l & 0 \\ 0 & r\end{pmatrix}^{-1}
      (\dd \tilde{\Phi}\trans \dd \Phi)^{-1}
      \begin{pmatrix} l & 0 \\ 0 & r\end{pmatrix}^{-1}
      \begin{pmatrix} f \\ 0\end{pmatrix}
      \\
      &=
      \underbrace{\dd \tilde{\Phi} (\dd \tilde{\Phi}\trans \dd \Phi)^{-1}}_{=\mathrlap{(\grad \tilde{\Phi})^{\dagger}}}
      \begin{pmatrix} \frac{1}{l} f \\ 0\end{pmatrix}
      = \frac{1}{\sqrt{\det (\dd \tilde{\Phi}\trans \dd \tilde{\Phi})}} \frac{1}{l}f \tau_{\textup{long}}
      \\
      &= r \frac{1}{\sqrt{\det (\dd {\Phi}\trans \dd {\Phi})}} f \tau_{\textup{long}},    
    \end{align*}
    which finishes the proof.
    \qedhere
  \end{steps}
\end{proof}

\begin{proposition}\label{prop:adj}
The operators in \eqref{eq:magop} and \eqref{eq:elop} fulfill
\[
  \portOp_\textup{el}\adjun=\portOp_\textup{mag}.
\]
 %
\end{proposition}

\begin{proof}
It can be seen that the desired statement is equivalent to
\[
    \portOp_{i,\textup{mag}}\adjun=\portOp_{i,\textup{el}}\quad \forall\,\,i=1,\ldots,k,\]
    whence we prove the latter. Let $i\in\{1,\ldots,k\}$, $f\in \Lp{2}((0,1))$, $g\in \Lp{2}(\Gamma_{i,\textup{lat}})$. 
  By \Cref{th:explicit-form-of-P} we have $\portOp_{i,\textup{el}}f = \frac{r_{i}}{\sqrt{\det \dd \Phi_{i}\trans \dd \Phi_{i}}}f\circ \Psi_{i} \tau_{\textup{long}}$. Then the desired result follows from (for sake of brevity, we neglect the integration variables)
  \begin{align*}
    \scprod{\portOp_{i,\textup{el}}f}{g}_{\Lp{2}(\Gamma_{i,\textup{lat}})}
    &= \int_{\Gamma_{\textup{lat}}} \frac{r_{i}}{\sqrt{\det \dd \Phi_{i}\trans \dd \Phi_{i}}} f \circ \Psi_{i} \tau_{\textup{long}} \cdot \conj{g} \dx[\xi]
    \\
    &= \int_{0}^{1} \int_{-\uppi}^{\uppi} r_{i} f \tau_{\textup{long}} \cdot \conj{g} \circ \Phi_{i} \dx[\theta] \dx[\eta]
    \\
    &= \int_{0}^{1} f \int_{-\uppi}^{\uppi} r_{i} \underbrace{\tau_{\textup{long}} \times \nu}_{=\mathrlap{\frac{1}{r_{i}}\beta_{i\eta}^{\prime}}}\mathclose{} \cdot\mathopen{} (\conj{g} \times \nu) \circ \Phi_{i} \dx[\theta] \dx[\eta]
    \\
    &= \int_{0}^{1} f \conj{\oint_{\alpha_{i}(\eta) + \beta_{i\eta}(\argdot)} g \times \nu \cdot \dx[\bm{s}]} \dx[\eta]\\[2mm]
    &= \scprod{f}{\portOp_{i,\textup{mag}}g}_{\Lp{2}((0,1))}.
    \qedhere
  \end{align*}
\end{proof}

\subsection{Power balance}
We now consider the power balance of the overall coupled system. To this end, we combine our results on the power balances of the transmission line and the electromagnetic field (see \Cref{sec:enbaltl} and \Cref{sec:enbalem}) with the previously derived coupling relations and the established relation between them.
Unsurprisingly, the total energy consists of the sum of the energies of the transmission line and the electromagnetic field, i.e.,
\begin{align*}
 \MoveEqLeft
 \mathcal{E}\big(\bm{\psi}(t),\bm{q}(t),\bm{B}(t),\bm{D}(t)\big)\\
 &=\mathcal{E}_{\rm tl}(\bm{\psi}(t),\bm{q}(t))+\mathcal{E}_{\rm em}(\bm{B}(t),\bm{D}(t))\\
  &=\frac12\int_{0}^1\bm{\psi}(t,\eta)\adjun \bm{L}(\eta)^{-1} \bm{\psi}(t,\eta)+\bm{q}(t,\eta)\adjun \bm{C}(\eta)^{-1} \bm{q}(t,\eta) \dx[\eta]\\&\quad+\frac12\int_{\Omega}\bm{B}(t,\xi)\adjun \bm{\mu}(\xi)^{-1} \bm{B}(t,\xi)+\bm{D}(t,\xi)\adjun \bm{\epsilon}(\xi)^{-1} \bm{D}(t,\xi) \dx[\xi].
\end{align*}
Now, using the respective power balances, we obtain that
\begin{align*}
\ddts\mathcal{E}\big(\bm{\psi}(t),\bm{q}(t),\bm{B}(t),\bm{D}(t)\big)
 &=\ddts\mathcal{E}_{\rm tl}(\bm{\psi}(t),\bm{q}(t))+\ddts\mathcal{E}_{\rm em}(\bm{B}(t),\bm{D}(t))\\
  &\leq \Re\big( u(t)\adjun y(t)\big)+\Re\int_{0}^1\bm{I}_{\ext}(t,\eta)\adjun \bm{E}_{\ext}(t,\eta) \dx[\eta]\\&\quad+\Re\int_{\Gamma_{\textup{lat}}}\big(\tantr(\xi)\bm{E}(t,\xi)\big)\adjun  \big(\nu(\xi)\times\bm{H}(t,\xi)\big) \dx[\xi].
\end{align*}
Now using the coupling relations \eqref{eq:couplcondop}, we obtain that 
\begin{align*}
    \MoveEqLeft
    \Re\int_{0}^1\bm{I}_{\ext}(t,\eta)\adjun \bm{E}_{\ext}(t,\eta) \dx[\eta]+\Re\int_{\Gamma_{\textup{lat}}}\big(\tantr(\xi)\bm{E}(t,\xi)\big)\adjun  \big(\nu(\xi)\times\bm{H}(t,\xi)\big) \dx[\xi]\\
    &=\Re\scprod{\bm{I}_{\ext}(t)}{\bm{E}_{\ext}(t)}_{\Lp{2}((0,1);\C^k)}+\Re\scprod{\nu\times \bm{H}_{\ext}(t)}{\tantr\bm{E}(t)}_{\Lp{2}(\Gamma_{\textup{lat}};\C^{3})}\\
        &=-\Re\scprod{\portOp_\textup{mag} \big(\nu \times \bm{H}(t)\big)}{\bm{E}_{\ext}(t)}_{\Lp{2}((0,1);\C^k)}\\&\qquad+\Re\scprod{\nu\times \bm{H}_{\ext}(t)}{\portOp_\textup{mag}\bm{E}_{\ext}(t)}_{\Lp{2}(\Gamma_{\textup{lat}};\C^{3})}.
\end{align*}
Since, by Proposition~\ref{prop:adj}, $\portOp_\textup{el}\adjun=\portOp_\textup{mag}$, we obtain that the whole expression vanishes. Therefore, the power balance reads
\[\ddts\mathcal{E}\big(\bm{\psi}(t),\bm{q}(t),\bm{B}(t),\bm{D}(t)\big)\leq \Re \big(u(t)\adjun y(t)\big).\]
That is, $\Re \big(u(t)\adjun y(t)\big)$ stands for the power supplied at the ends of the cables.
A careful examination of the respective power balances in \Cref{sec:enbaltl} and \Cref{sec:enbalem} reveals that the total power loss, that is, the terms responsible for the inequality, is given by the sum of the expressions in \eqref{eq:distrtlloss}, \eqref{eq:endtlloss}, and \eqref{eq:emloss}. These represent, respectively, the power dissipated along the cables, at the cable ends, and in the electromagnetic field.
\begin{remark}
It follows from \Cref{def:coloc} that one may choose $m=2k$ and $W_B,W_C\in\C^{2k\times 4k}$ with
\[
  W_B = [\,\id_{2m},\,0\,], \qquad W_C = [\,0,\,\id_{2m}].
\]
In this case, the inputs and outputs consist of the voltages and currents at the cable ends, respectively. More precisely,
\[
  u(t)=\begin{pmatrix}
    \bm{V}(t,0) \\
    \bm{V}(t,1)
  \end{pmatrix},\quad
  y(t)=\begin{pmatrix}
    \phantom{-}\bm{I}(t,0) \\
    -\bm{I}(t,1)
  \end{pmatrix}.
\]
With this choice, the previously derived power balance takes the form
\[
  \ddts \mathcal{E}\big(\bm{\psi}(t),\bm{q}(t),\bm{B}(t),\bm{D}(t)\big)
  \leq \Re\big(\bm{V}(t,0)\bm{I}(t,0)\big) - \Re\!\big(\bm{V}(t,1)\bm{I}(t,1)\big).
\]
Since \eqref{eq:col2} becomes an equality for the above choice of $W_B$ and $W_C$, no dissipation occurs at the cable boundaries. Hence, the only contributions to the inequality arise from
\eqref{eq:distrtlloss} and \eqref{eq:emloss}. In particular, if $\bm{R}$, $\bm{G}$, and $\bm{\sigma}$ vanish identically, the system is lossless in the sense that
\[
  \ddts \mathcal{E}\big(\bm{\psi}(t),\bm{q}(t),\bm{B}(t),\bm{D}(t)\big)
  = \Re\big(\bm{V}(t,0)\bm{I}(t,0)\big) - \Re\big(\bm{V}(t,1)\bm{I}(t,1)\big).
\]
\end{remark}

\begin{remark}
The derived power balance shows that the coupling between the transmission line and the electromagnetic field is energetically consistent. That is, the power extracted from the transmission line corresponds to the power supplied to the electromagnetic field, and vice versa. Without going into further detail here, this corresponds to the coupling of port-Hamiltonian systems as described, for example, in \cite{CervdS07}.
\end{remark}

\section{Conclusion}\label{sec:conclusion}

We have presented a modeling approach for cable harnesses that leads to a coupled system of telegrapher’s equations and Maxwell’s equations, linked through boundary conditions. Curved cables have been taken into account, and appropriate coupling conditions have been derived. It is shown that the overall system satisfies a~power balance.


\iftoggle{springer}{\bmhead{Acknowledgments}}{}%
\iftoggle{default}{\par\smallskip\noindent\textbf{Acknowledgments.}}{}%
\iftoggle{birk}{\noindent\paragraph{Acknowledgments.}}{}
We thank Alexander Wierzba (U Twente) for the discussions about the physical rationale of the coupling.\\
This work was supported by the collaborative research center SFB 1701 ``Port-Hamiltonian Systems''.



\begin{appendices}

\section{Integral details}

We first show that the port operator $\portOp_{i,\textup{mag}}$ as introduced in \Cref{sec:analysis-of-port-operator} is well-defined.

\begin{lemma}\label{le:port-operator-well-defined}
  With $\Gamma_{\textup{lat}}$ as in \Cref{sec:Omega}, the operator $\portOp_{i,\textup{mag}}$ as in \eqref{eq:magiop} is well-defined.
\end{lemma}

\begin{proof}
For notational simplicity we further leave out the indices $i$ for the functions introduced in \Cref{sec:Omega}, i.e., we have just $\alpha$, $\beta_{\eta}$, $\Phi$, etc.\ instead of $\alpha_{i}$, $\beta_{i\eta}$, $\Phi_{i}$, etc. First recall
that
\begin{align*}
  \oint_{\alpha(\eta) + \beta_{\eta}} f(s) \cdot \dx[\bm{s}]
  &= \int_{-\uppi}^{\uppi}  f\big(\underbrace{\alpha(\eta) + \beta_{\eta}(\theta)}_{=\mathrlap{\Phi(\eta,\theta)}}\big) \cdot \beta_{\eta}'(\theta) \dx[\theta] \\
  &= \int_{-\uppi}^{\uppi} (f\circ \Phi)(\eta,\theta) \cdot\beta_{\eta}^{\prime}(\theta) \dx[\theta].
\end{align*}
Note that $f \in \Lp{2}(\Gamma_{\textup{lat}})$ is equivalent to $f \circ \Phi \in \Lp{2}\bigl((0,1) \times (-\uppi,\uppi)\bigr)$ and implies $(f \circ \Phi) \cdot \beta_{\eta}^{\prime}\in \Lp{2}\bigl((0,1) \times (-\uppi,\uppi)\bigr)$ and therefore also in $\Lp{1}\bigl((0,1) \times (-\uppi,\uppi)\bigr)$. Hence, Fubini's theorem \cite[Thm.~A.6.10]{Alt16} gives
  \begin{align*}
    \Bigl(\eta \mapsto \int_{-\uppi}^{\uppi} \abs{(f \circ \Phi)(\eta,\theta) \cdot \beta_{\eta}^{\prime}(\theta)}^{2} \dx[\theta] \Bigr) &\in \Lp{1}\big((0,1)\big) \\
    \mathllap{\text{and}\quad}
    \Bigl(\eta \mapsto \int_{-\uppi}^{\uppi} (f \circ \Phi)(\eta,\theta) \cdot \beta_{\eta}^{\prime}(\theta) \dx[\theta] \Bigr) &\in \Lp{1}\big((0,1)\big).
  \end{align*}
  In particular, the latter function is measurable.
  Cauchy--Schwarz (applied on $(f\circ \Phi)(\eta,\theta) \cdot \beta_{\eta}^{\prime}(\theta)$ and $1$ for the $\theta$ integral) finally gives
  \begin{align*}
    \Bigl(\eta \mapsto \int_{-\uppi}^{\uppi} (f \circ \Phi)(\eta,\theta) \cdot \beta_{\eta}^{\prime}(\theta) \dx[\theta] \Bigr) \in \Lp{2}\big((0,1)\big).
  \end{align*}
  Therefore, the operator $\portOp_{i,\textup{mag}}$ is well-defined.
\end{proof}

\section{Continuous parametrization of the normal plane}

Note that we do not use the Frenet--Serret formulas to parameterize the tangential and normal vectors of the path, because this description can be discontinuous in points where the curvature vanishes. Even if the curvature does not vanish we would still need $\conC^{3}$ regularity of the path $\alpha$ for the Frenet--Serret formulas to be $\conC^{1}$.

\begin{lemma}\label{le:normal-fields-of-path}
  Let $l>0$, $\alpha \in \conC^{2}([0,1], \R^{3})$ with $\norm{\alpha'(\eta)}=l$ for all $\eta\in[0,1]$. Then there exist $\kappa_{1}, \kappa_{2} \in \conC^{1}([0,1], \R^{3})$ such that, for all $\eta\in[0,1]$, $(\tfrac1l\alpha'(\eta),\kappa_{1}(\eta),\kappa_{2}(\eta))$ is an orthonormal basis of $\R^3$ with
  \begin{equation*}
    \forall\,\eta\in[0,1]:\quad
    \det[\alpha'(\eta),\kappa_{1}(\eta),\kappa_{2}(\eta)] = l.
  \end{equation*}
  In particular, for all $\eta\in[0,1]$, $\kappa_{1}(\eta)$ and $\kappa_{2}(\eta)$ span the normal space of $\alpha$ at $\eta$.
\end{lemma}

\begin{proof}
  Note that $\Lambda \coloneqq \dset{\tfrac{1}{l}\alpha'(\eta)}{\eta\in[0,1]}$, fulfills $\Lambda\subset\mathbb{S}_2$, where the latter denotes the unit sphere in $\R^3$. The definition of $\Lambda$ gives
  \begin{equation*}
    \uplambda^{1}(\Lambda)\leq \max_{\eta\in[0,1]}\frac{\|\alpha''(\eta)\|}l,
  \end{equation*}
  where $\uplambda^{1}$ denotes the one-dimensional Lebesgue measure.\footnote{Note that $\uplambda^{1}$ is actually a measure on $\R$. To be precise we actually use the one-dimensional Hausdorff measure.} Since $\uplambda^{1}(\mathbb{S}_2)=\infty$, we can conclude that
  there exists some $w\in \mathbb{S}_2$ with $w\neq \alpha'(\eta)$ and $w\neq-\alpha'(\eta)$ for all $\eta\in[0,1]$. Consequently, the cross product fulfills $\alpha'(\eta)\times w\neq0$ for all $\eta\in[0,1]$, and the choice 
 %
  \begin{align*}
    n_{1}(\eta) = \frac{\alpha'(\eta) \times w}{\norm{\alpha'(\eta) \times w}}
    \quad\text{and}\quad
    n_{2}(\eta) = \frac{n_{1}(\eta) \times \alpha'(\eta)}l.
  \end{align*}
  leads to a~pointwise orthonormal basis $(\tfrac{1}{l}\alpha',n_{1},n_{2})$ of $\R^{3}$. By further using the continuity of $\alpha',n_{1},n_{2}\colon [0,1] \to \R^{2}$, we have that $d\colon [0,1]\to\R$ with $d(\eta)\coloneqq \det[\alpha'(\eta),n_{1}(\eta),n_{2}(\eta)]$ is constant with either $d\equiv l$ or $d\equiv-l$. Now choosing
  \begin{equation*}
    \kappa_1 \coloneqq
    \begin{cases}
      n_1,&\text{if }d\equiv\phantom{-}l,\\
      n_2,&\text{if }d\equiv-l,
    \end{cases}
    \quad
    \kappa_2 \coloneqq
    \begin{cases}
      n_2,&\text{if }d\equiv\phantom{-}l,\\
      n_1,&\text{if }d\equiv-l,
    \end{cases}
  \end{equation*}
  we obtain functions with desired properties.
\end{proof}

\begin{lemma}\label{th:linear-dependencies-of-derivatives-of-coordinates}
  Let $\alpha \in \conC^{2}([0,1];\R^{3})$ a path such that $\norm{\alpha^{\prime}(\eta)} = l$ for all $\eta \in [0,1]$ and $\kappa_{1},\kappa_{2} \in \conC^{1}([0,1];\R^{3})$ such that $\norm{\kappa_{1}}$, $\norm{\kappa_{2}}$ are constant, $(\frac{1}{l}\alpha^{\prime},\kappa_{1},\kappa_{2})$ is an orthogonal basis with
  $\det [\alpha^{\prime},\,\kappa_{1},\, \kappa_{2}] > 0$ (e.g., as in \Cref{le:normal-fields-of-path}).
  Then $\kappa_{1}^{\prime}$ is in the span of $(\alpha^{\prime}, \kappa_{2})$, and $\kappa_{2}^{\prime}$ is in the span of $(\alpha^{\prime}, \kappa_{1})$.
\end{lemma}

\begin{proof}
  Note that
  \begin{equation*}
    \alpha^{\prime} \cdot \alpha^{\prime\prime} = \frac{1}{2} (\alpha^{\prime} \cdot \alpha^{\prime})^{\prime} = 0
  \end{equation*}
  and similar for $\kappa_{1}, \kappa_{2}$ we get
  \begin{equation*}
    \kappa_{1}\cdot \kappa_{1}^{\prime} = 0
    \quad\text{and}\quad
    \kappa_{2}\cdot \kappa_{2}^{\prime} = 0.
  \end{equation*}
  Hence, $\alpha^{\prime\prime} \perp \alpha^{\prime}$, $\kappa_{1}^{\prime} \perp \kappa_{1}$ and $\kappa_{2}^{\prime} \perp \kappa_{2}$.
  Moreover, $c_{2} \kappa_{2} = \frac{1}{l} \alpha^{\prime} \times \kappa_{1}$
  and $c_{1} \kappa_{1} = \kappa_{2} \times \alpha^{\prime}$
  for some $c_{1},c_{2} > 0$
  follows from $\det [\alpha^{\prime},\,\kappa_{1},\, \kappa_{2}] > 0$.
  By the product rule we have for some $d_{1},d_{2} \in \R$
  \begin{equation*}
    c_{1} \kappa_{1}^{\prime} = (\kappa_{2} \times \alpha^{\prime})^{\prime}
    = \underbrace{\kappa_{2}^{\prime} \times \alpha^{\prime}}_{=\mathrlap{d_{1} \kappa_{2}}}
    + \underbrace{\kappa_{2} \times \alpha^{\prime\prime}}_{=\mathrlap{d_{2} \alpha^{\prime}}},
  \end{equation*}
  because both $\alpha^{\prime}$ and $\kappa_{2}^{\prime}$ are orthogonal to $\kappa_{2}$ by the first part and analogously for the second vector cross product.
  In the same way we derive that $\kappa_{2}^{\prime}$ is in the span of $(\alpha^{\prime}, \kappa_{1})$.
\end{proof}

\end{appendices}


\bibliographystyle{abbrvurl}
\bibliography{references}{}

\end{document}